\documentclass[a4paper,reqno,11pt]{amsart}
\usepackage{amsmath,amssymb}
\usepackage{amsthm}
\usepackage{amscd}
\usepackage{amssymb}
\usepackage{amsfonts}
\usepackage{latexsym}
\usepackage[mathscr]{eucal}

\newcommand{\Ad}{\mathrm{Ad}}

\newcommand{\C}{\mathrm{C}}
\newcommand{\dab}{d(A,B)}
\newcommand{\w}{\mathrm{w}}
\newcommand{\id}{\mathrm{id}}

\newtheorem{theorem}{Theorem}
\newtheorem{thm}{Theorem}[section]
\newtheorem{prop}[thm]{Proposition}
\newtheorem{lemma}[thm]{Lemma}
\newtheorem{cor}[thm]{Corollary}
\newtheorem{defn}[thm]{Definition}
\newtheorem{exam}[thm]{Example}

\newtheorem{rem}[thm]{Remark}

\begin{document}
\allowdisplaybreaks[3] 

\title{Perturbations of crossed product C$^*$-algebras by amenable groups}

\author[S. Ino]{SHOJI INO}

\address{Department of Mathematical Sciences, Kyushu University, Motooka, Fukuoka, 819-0395, Japan.}

\email{s-ino@math.kyushu-u.ac.jp}


\maketitle

\begin{abstract}
We study uniform perturbations of crossed product C$^*$-algebras by amenable groups. 
Given a unital  inclusion of C$^*$-algebras $C\subseteq D$
and sufficiently close separable intermediate C$^*$-subalgebras $A$, $B$ for this inclusion with a conditional expectation from $D$ onto $B$,
if $A=C\rtimes G$ with $G$ discrete amenable, then $A$ and $B$ are isomorphic.
Furthermore, if $C\subseteq D$ is irreducible, then $A=B$.
\end{abstract}

\section{introduction}
Kadison and Kastler started the study of perturbation theory of operator algebras with \cite{KK} in 1972.
They equipped the set of operator algebras on a fixed Hilbert space 
with a metric induced by Hausdorff distance between the unit balls.
Examples of close operator algebras are obtained by conjugating by a unitary near to the identity.
They conjectured sufficiently close operator algebras must be unitarily equivalent.
For injective von Neumann algebras, this conjecture was settled in \cite{Chris2, RT, Johnson1, Chris5}
with earlier special cases \cite{Chris1, P1}.
Cameron et al. \cite{CCSSWW} and Chan \cite{Chan} gave examples of non-injective von Neumann algebras satisfying the Kadison-Kastler conjecture.
In Christensen \cite{Chris3}, this conjecture was solved positively for von Neumann subalgebras of a common finite von Neumann algebra.

For C$^*$-algebras, the separable nuclear case was solved positively in Christensen et al. \cite{CSSWW}, 
building on the earlier special cases in \cite{Chris5, PR1, PR2, Khoshkam1}.
In full generality, there are examples of arbitrarily close non-separable nuclear C$^*$-algebras which are not $*$-isomorphic in Choi and Christensen \cite{CC}.
Johnson gave examples of arbitrarily close pairs of separable nuclear C$^*$-algebras which conjugate by unitaries where the implementing unitaries could not be chosen to be the identity in \cite{Johnson2}.

The author and Watatani \cite{IW} showed that for an inclusion of simple C$^*$-algebras $C\subseteq D$ with finite index in the sense of Watatani \cite{Watatani},
sufficiently close intermediate C$^*$-subalgebras are unitarily equivalent.
The implementing unitary can be chosen close to the identity and in the relative commutant algebra $C'\cap D$.
Our estimates depend on the inclusion $C\subseteq D$, since we use the finite basis for $C\subseteq D$.
Dickson  obtained uniform estimates independent of all inclusions in \cite{Dickson}.
To get this, Dickson showed that row metric is equivalent to the Kadison-Kastler metric.

The author \cite{I} showed that von Neumann subalgebras of a common von Neumann algebra with finite probabilistic index in the sense of Pimsner-Popa \cite{PP}
satisfy the Kadison-Kastler conjecture.
The implementing unitary can be chosen as being close to the identity.
Compared with the author and Watatani case \cite{IW},
we do not assume that von Neumann subalgebras have a common subalgebra with finite index.

In this paper, we study perturbations of crossed product C$^*$-algebras by discrete amenable groups.
We introduce crossed product-like inclusions of C$^*$-algebras in Definition \ref{amenable}.
For a unital inclusion of $\C^*$-algebras $A\subseteq B$,
we call $A\subseteq B$ is crossed product-like if there exists a discrete group $U$ in the normalizer $\mathcal{N}_B(A)$ of $A$ in $B$ such that
$A$ and $U$ generate $B$.
An example of crossed product-like inclusions is $A\subseteq A\rtimes G$, where $G$ is a discrete group.

Now suppose that we have a unital inclusion $C\subseteq D$ of $\C^*$-algebras and 
two close separable intermediate $\C^*$-subalgebras $A,B$ for  this inclusion.
If there is a conditional expectation $E\colon D\to B$,
then we get a map from $A$ into $B$ which is uniformly close to the identity map of $A$ by restricting $E$ to $A$.
Since $C$ is a subalgebra of $A\cap B$, $E|_A\colon A\to B$ is a $C$-fixed map, that is, $E|_A(c)=c$ for any $c\in C$.
Furthermore, if $C\subseteq A$ is crossed product-like by a discrete amenable group $U$ in $\mathcal{N}_A(C)$, then we can consider the point-norm averaging technique form \cite{CSSWW} by using the amenability of $U$.
To apply this technique to $E|_A$ we need that $E|_A$ is a $C$-fixed map.
Then in Lemma \ref{1.4}, we can obtain a $C$-fixed $(X,\varepsilon)$-approximate $*$-homomorphism from $A$ into $B$ for a finite subset $X$ in $A_1$ and $\varepsilon>0$.
To show this, we modify \cite[Lemma 3.2]{CSSWW} to $C$-fixed versions.
In Lemma \ref{1.5}, we obtain unitaries which conjugate these maps by modifying \cite[Lemma 3.4]{CSSWW} to $C$-fixed version.
The unitaries can be chosen in the relative commutant $C'\cap D$ of $C$ in $D$.
Therefore, if $C\subseteq D$ is irreducible, then the untiaries are scalars.
Then 
by these lemmas, we show our first main result: Theorem A, which is appeared in Theorem \ref{irreducible}.

\begin{theorem}
Let $C\subseteq D$ be a unital irreducible inclusion of $\C^*$-algebras 
acting on a separable Hilbert space $H$.
Let $A$ and $B$ be separable intermediate $\C^*$-subalgebras for $C\subseteq D$ with 
a conditional expectation from $D$ onto $B$.
Suppose that 
$C\subseteq A$ is crossed product-like by a discrete amenable group
and $\dab<140^{-1}$.
Then $A  = B$.
\end{theorem}

In Theorem B,
we show our second main result.
By an intertwining argument which is modified \cite[Lemma 4.1]{CSSWW} to $C$-fixed version, we show that $A$ is $*$-isomorphic to $B$.
The implementing surjective $*$-isomorphism can be chosen as $C$-fixed.
Theorem B is provided in Section \ref{isomorphism} as Theorem \ref{3.3}.

\begin{theorem}
Let $C\subseteq D$ be a unital inclusion of $\C^*$-algebras  and 
let $A$ and $B$ be separable intermediate $\mathrm{C}^*$-subalgebras for $C\subseteq D$ with a conditional expectation from $D$ onto $B$.
Suppose that $C\subseteq A$ is crossed product-like by a discrete amenable group
and $d(A,B)<10^{-3}$.
Then
there exists a $C$-fixed surjective $*$-isomorphism $\alpha \colon A\to B$.
\end{theorem}

In section \ref{von},
we consider crossed product-like inclusions of von Neumann algebras.
Given an inclusion $N\subseteq M$ of von Neumann algebras, we call $N\subseteq M$ is crossed product-like if there is a discrete group $U$ in $\mathcal{N}_M(N)$ such that $M$ is generated by $N$ and $U$.
For a  crossed product-like inclusion $A\subseteq B$ of $\C^*$-algebras
acting non-degenerately on $H$,
 an  inclusion $\overline{A}^{\w}\subseteq \overline{B}^{\w}$ of von Neumann algebras is crossed product-like.
In Theorem C,
we consider the perturbations of crossed product von Neumann algebras 
by discrete amenable groups.
This result is based on Christensen's work \cite{Chris2} and is appeared in Theorem \ref{2.3.5}.

\begin{theorem}
Let $N\subseteq M$ be an inclusion of von Neumann algebras in $\mathbb{B}(H)$ and 
let $A,B$ be intermediate von Neumann subalgebras for $N \subseteq M$
with a normal conditional expectation from $M$ onto $B$. 
Suppose that $N\subseteq A$ is crossed product-like by a discrete amenable group
and $d(A,B)<\gamma<10^{-2}$.
Then there exists a unitary $u\in N ' \cap (A\cup B)''$ 
such that $u A u^*= B$ 
and $\| u - I \| \le 2(8+\sqrt{2})\gamma$.
\end{theorem}

By the theorem above,
we can consider the perturbations of the second dual C$^*$-algebras of crossed product algebras by amenable groups in Corollary \ref{2.4}.
Given a unital inclusion $C\subseteq D$ of $\C^*$-algebras 
and sufficiently close intermediate $\C^*$-subalgebras $A,B$ for this inclusion,
if $C\subseteq A$ is a crossed product-like inclusion by a discrete amenable group and there is a conditional expectation $E\colon D\to B$,
then $A^{**}$ and $B^{**}$ are unitarily equivalent.
To show this, we use a normal conditional expectation $E^{**}\colon D^{**}\to B^{**}$ and 
identify $A^{**},B^{**},C^{**}$ and $D^{**}$ with $\pi(A)'',\pi(B)'',\pi(C)''$ and $\pi(D)''$,
respectively, where $\pi$ is the universal representation of $D$.

In Proposition \ref{4.3},
we obtain a unitary such that the unitary implement a $*$-isomorphism under the assumption 
$C'\cap C^*(A,B)\subseteq \overline{C'\cap A}^{\w}$.
To show Proposition \ref{4.3}
we prepare Lemma \ref{4.1} and \ref{4.2} by using Lemma \ref{1.5} and \ref{1.8} and Theorem \ref{3.3}.
Combining Proposition \ref{4.3} with Corollary \ref{2.4} gives Theorem D, which is appeared in Theorem \ref{main}.
To show this, we modify the arguments of Section 5 in Christensen et al. \cite{CSSWW}.

\begin{theorem}
Let $C\subseteq D$ be a unital inclusion of $\mathrm{C}^*$-algebras 
acting on a separable Hilbert space $H$.
Let $A$ and $B$ be separable intermediate $\C^*$-subalgebras for $C\subseteq D$ with 
a conditional expectation $E\colon D\to B$.
Suppose that $C\subseteq A$ is crossed product-like by a discrete amenable group
and $C'\cap A$ is weakly dense in $C'\cap \overline{A}^{\w}$.
If $\dab<10^{-7}$,
then there exists a unitary $u\in C'\cap (A\cup B)''$ such that $u A u^* = B$.
\end{theorem}


\section{Preliminaries}


Given a C$^*$-algebra $A$, we denote by $A_1$ and $A^u$ the unit ball of $A$ and the unitaries in $A$, respectively.
We recall Kadison and Kastler's metric on the set of all C$^*$-subalgebras of a C$^*$-algebra from \cite{KK}.

\begin{defn}\upshape
Let $A$ and $B$ be C$^*$-subalgebras of a C$^*$-algebra $C$.
Then we define a metric between $A$ and $B$ by
\[
d(A,B):= \max \left\{ \sup_{a\in A_1} \inf_{b\in B_1} \| a - b\| , \ \sup_{b\in B_1}\inf_{a\in A_1} \| a-b\|     \right\}.
\]
\end{defn}

In the  definition above,
$d(A,B)<\gamma$ if and only if for any $x$ in either $A_1$ or $B_1$,
there exists $y$ in other one such that $\| x - y\| <\gamma$.

\begin{exam}\upshape
Let $A$ be a C$^*$-algebra in $\mathbb{B}(H)$ and $u$ be a unitary in $\mathbb{B}(H)$.
Then $d(A, u A u^*)\le 2\| u-I_H\|$.
\end{exam}

Near inclusions of $\C^*$-algebras are defined by Christensen in \cite{Chris5}.

\begin{defn}\upshape
Let $A$ and $B$ be C$^*$-subalgebras of a C$^*$-algebra $C$ and let $\gamma>0$.
We write $A\subseteq_{\gamma}B$ if for any $x\in A_1$ there exists $y\in B$ such that $\| x-y\| \le \gamma$.
If there is $\gamma'<\gamma$ with $A\subseteq_{\gamma'}B$, then we write $A\subset_{\gamma}B$.
\end{defn}

The next two proposition is folklore.
The second can be found as \cite[Proposition 2.10]{CSSWW}.

\begin{prop}\label{surjective}
Let $A$ and $B$ be $\C ^*$-algebras with $A\subseteq B$.
If $B\subset_1 A$, then $A=B$.
\end{prop}

\begin{prop}
Let $A$ and $B$ be $\C^*$-subalgebras of a $\C^*$-algebra $C$.
If $B\subset_{1/2} A$ and $A$ is separable,
then $B$ is separable.
\end{prop}

The following lemma appears in \cite[Lemma 5]{KK}.

\begin{lemma}\label{weak-closure}
Let $A$ and $B$ be $\C^*$-subalgebras acting on a Hilbert space $H$.
Then $d(\overline{A}^{\w},\overline{B}^{\w})\le \dab$.
\end{lemma}

The lemma below shows some standard estimates.

\begin{lemma}\label{polar}
Let $A$ be a unital $\C^*$-algebra.
\begin{enumerate}
\item
Given $x\in A$ with $\| I-x \|<1$, let $u\in A$ be the unitary in the polar decomposition $x=u|x|$.
Then,
\[
\| I-u\| \le \sqrt{2}\| I-x \|.
\]

\item
Let $p\in A$  be a projection and $a\in A$ a self-adjoint operator.
Suppose that $\delta:=\| a-p\| <1/2$.
Then $q:=\chi_{[1-\delta,1+\delta]}(a)$ is a projection in $C^*(a,I)$ satisfying 
\[
\| q-p\|\le 2\| a-p\| <1.
\]

\item
Let $p,q\in A$ be projections with $\| p-q\| <1$.
Then there exists a unitary $w\in A$ such that
\[
w p w^* =q \ \ \text{and} \ \ \| I-w \| \le \sqrt{2}\| p-q\|.
\]
\end{enumerate}
\end{lemma}

In the paper, we consider metric between maps restricted to finite sets.
The following are introduced by \cite{CSSWW}.

\begin{defn}\upshape
Let $A$ and $B$ be C$^*$-algebras and let $\phi_1,\phi_2\colon A\to B$ be maps.
Given a subset $X\subseteq A$ and $\varepsilon>0$,
write $\phi_1\approx_{X,\varepsilon}\phi_2$ if
\[
\|\phi_1(x)-\phi_2(x)\|\le \varepsilon, \ \ x\in X.
\]
\end{defn}

\begin{defn}\upshape
Let $A$ and $B$ be C$^*$-algebras, $X$ a subset of $A$ and $\varepsilon>0$.
Given a completely positive contractive map (cpc map) $\phi\colon A\to B$,
we call $\phi$ is an ($X,\varepsilon$)-{\it approximate} $*$-{\it homomorphism}
if it satisfies
\[
\| \phi (x) \phi(x^*)-\phi(xx^*)\| \le \varepsilon, \ \ x\in X\cup X^*.
\]
\end{defn}

We only consider pairs of the form $(x,x^*)$ in the previous definition
by the following proposition, which can be found as \cite[Lemma 7.11]{Paulsen}.

\begin{prop}\label{1/2}
Let $A$ and $B$ be $\mathrm{C}^*$-algebras and $\phi\colon A\to B$ a cpc map.
Then for $x,y\in A$,
\[
\| \phi(x y) -\phi(x) \phi(y) \| \le \| \phi(x x^*) -\phi(x) \phi(x^*) \|^{1/2} \|y \|.
\]
\end{prop}

\begin{defn}\upshape
Let $A$ and $B$ be C$^*$-algebras and let $C$ be a C$^*$-subalgebras of $A\cap B$.
A map $\phi:A\to B$ is  {\it $C$-fixed} if $\phi|_C=\id_C$.
\end{defn}

\begin{rem}\upshape
Given a map $\phi\colon A\to B$ between $\C^*$-algebras and a $\C^*$-subalgebra $C$  of $A\cap B$,
if $\phi$ is $C$-fixed, then $\phi$ is a $C$-bimodule map, that is,
for $x,z\in C$ and $y\in A$,
\[
\phi(x y z)=x \phi(y) z.
\]
\end{rem}

The following lemma appears in \cite[p.332]{Arveson}.
We need the lemma in Lemma \ref{1.4}, \ref{1.5} and \ref{1.8}.

\begin{lemma}\label{Arveson}
Let $X\subseteq \mathbb{C}$ be a compact set and $\varepsilon,M>0$.
Then given a continuous function $f\in C(X)$, 
there exists $\eta>0$
such that
for any Hilbert space $H$, normal operator $s\in\mathbb{B}(H)$ with $\mathrm{s p}(s)\subseteq X$ and $a\in \mathbb{B}(H)$ with $\| a \| \le M$,
the inequality $\| s a -  a s \| <\eta$ implies that $\|   f(s)a  - a f(s) \| <\varepsilon$.
\end{lemma}

\begin{proof}
Let $p$ be a polynomial such that $\| f- p\|<\varepsilon/(4M)$, where this norm is the supremum norm of $C(X)$.
Let $p$ have the form  $p(t)= c_o+c_1 t +\cdots + c_n t^n$.
Define
\[
\eta:= \frac{\varepsilon}{2}  \left( \sum_{k=1}^n k |c_k|  \right)^{-1}.
\]
Let a Hilbert space $H$ be given and 
let a normal operator $s\in \mathbb{B}(H)$ and $a\in \mathbb{B}(H)_1$ satisfy $\| s a - a s \|<\eta$.
Let $D$ be the derivation $D(x)=x a - a x$.
Since $D(s^{n+1})=s D(s^n) - D(s) s^n$, $\| D(s^{n+1})\| \le \| D(s^n)\| + \| D(s)\|$.
Hence, $\| D(s^n) \| \le n \| D(s) \|$.
Therefore,
\begin{align*}
\| f(s) a - a f(s)\| 
&\le \| f(s) a - p(s) a\| +\| D(p(s))\| + \| a p(s) - a f(s)\| \\
&\le 2 \| f-p\| \| a\| + \sum_{k=1}^n k |c_k| \| D(s) \| <\varepsilon,
\end{align*}
and the lemma follows.
\end{proof}

The next lemma appears in the proof of \cite[Lemma 3.7]{CSSWW}.

\begin{lemma}\label{1.7}
Let $H$ be a Hilbert space.
Then for any $\mu_0>0$, there exists $\mu>0$ with the following property$\colon$
given a finite set $S\subseteq H_1$ and a self-adjoint operator $h\in \mathbb{B}(H)_1$,
there exists a finite set $S'\subseteq H_1$ such that for any self-adjoint operator $k\in \mathbb{B}(H)_1$, if 
\[
\| (h-k) \xi' \| < \mu , \ \ \xi' \in S',
\]
then
\[
\| ( e^{i\pi h} - e^{i\pi k}) \xi \| <\mu_0 \ \ and \ \ 
\| ( e^{i\pi h} - e^{i\pi k})^* \xi \| <\mu_0, \ \ \xi\in S.
\]
\end{lemma}

\begin{proof}
There exists a polynomial $p(t)=\sum_{j=0}^r \lambda_j t^j$ such that
\begin{align}\label{1.7.1}
| p(t) - e^{i\pi t} | <\frac{\mu_0}{3}, \ \ -1\le t \le 1.
\end{align}
Let
\begin{align}\label{1.7.2}
\mu:= \frac{\mu_0}{3 r \sum_{j=0}^r | \lambda_j | }.
\end{align}
Given a finite set $S\subseteq H_1$ and a self-adjoint operator $h \in \mathbb{B}(H)_1$,
define
\begin{align*}
S':= \{ h^m \xi : \xi\in S, \, m \le r-1 \}.
\end{align*}
Let $k \in\mathbb{B}(H)_1$ be a self-adjoint operator with 
\begin{align}\label{1.7.3}
\| (h-k) \xi'\| <\mu, \ \ \xi'\in S'.
\end{align}
For any $\xi \in S$ and $0\le j\le r$, by (\ref{1.7.3}),
\begin{align*}
\| (h^j - k^j) \xi \| 
&\le \| (h^j - k h^{j-1}) \xi \| + \| ( k h^{j-1} - k^2 h^{j-2}) \xi \| + \dots + \| (k^{j-1} h - k^j) \xi \| \\
&\le \| (h - k) h^{j-1} \xi \| + \| k ( h-k) h^{j-2} \xi \| + \dots + \| k^{j-1} (h - k ) \xi\| \\
&\le \sum_{m=0}^{j-1} \| (h - k ) h^m \xi \|  <  r \mu.
\end{align*}
Thus, for $\xi\in S$,
\begin{align*}
\| (p(h) - p(k)) \xi \|
\le \sum_{j=0}^r |\lambda_j| \| (h^j- k^j) \xi \| 
\le \sum_{j=0}^r |\lambda_j| r\mu = \frac{\mu_0}{3},
\end{align*}
by (\ref{1.7.2}).
\begin{align*}
\| ( e^{i\pi h} - e^{i\pi k})\xi \| 
&\le \| (e^{i\pi h} -p(h)) \xi\| + \| (p(h) -p(k) )\xi\| + \| (p(k) - e^{i\pi k}) \xi \| \\
&\le \frac{\mu_0}{3} + \frac{\mu_0}{3} + \frac{\mu_0}{3} = \mu_0,
\end{align*}
by (\ref{1.7.1}).
Similarly, we have $\|(e^{i\pi h}-e^{i\pi k})^* \xi\| <\mu_0$.
\end{proof}


\section{Crossed product-like inclusions and approximate averaging}


In this section,
we introduce the crossed product-like inclusions of C$^*$-algebras.
Moreover,
we use the F\o lner condition of discrete amenable groups to modify the averaging results in \cite[Section 3]{CSSWW}.
In Theorem \ref{irreducible}, we show our first main result: Theorem A.

Given an inclusion $A\subseteq B$ of C$^*$-algebras, we denote by $\mathcal{N}_B(A)$ the normalizer  of $A$ in $B$,
that is,
$\mathcal{N}_B(A)= \{ u \in B^u : u A u^*=A\}$.

\begin{defn}\label{amenable}\upshape
Let $A\subseteq B$ be a unital  inclusion of C$^*$-algebras.
Then we call the inclusion $A\subseteq B$ is {\it crossed product-like}
if 
there exists a discrete group $U$ in $\mathcal{N}_B(A)$ such that
$B= C^*(A,U)$.
\end{defn}

Since $U$ is in $\mathcal{N}_B(A)$,
$B=C^*(A,U)$ is the norm closure of $\mathrm{span}\{ a u : a\in A, u\in U\}$.
Throughout this paper,
we only consider crossed product-like inclusions are by discrete {\it amenable} groups.

\begin{rem}\upshape
For any $x\in B$ and $\varepsilon>0$,
there exist $\{ a_1,\dots,a_N \}\subseteq A_1$ and $\{u_1,\dots,u_N\}\subseteq U$ such that $\| x- \sum_{i=1}^N a_i u_i \| <\varepsilon$.
In fact, let $K$ be a positive integer with $K \ge \max\{ \| a_1\|,\dots,\|a_N\| \}$.
Define
\[
a_{(i-1)K+j}':= \frac{1}{K} a_i, \ \ i=1,2,\dots,N, j=1,2,\dots,K.
\]
Then $a_k'\in A_1$ and
\[
\sum_{i=1}^N a_i u_i = \sum_{j=1}^K \sum_{i=1}^N a_{(i-1)K+j}' u_i.
\]
\end{rem}

\begin{exam}\upshape
Let $G$ be a discrete amenable group acting on a $\C^*$-algebra $A$.
Then an  inclusion $A\subseteq A\rtimes G$ is crossed product-like by $\{\lambda_g\}_{g\in G}$.
\end{exam}

\begin{exam}\upshape
Let $(A,G,\alpha,\sigma)$ be a twisted $\C^*$-dynamical system and let $A\rtimes_{\alpha,r}^{\sigma} G$ be the reduced twisted crossed product.
Then an inclusion $A\subseteq A\rtimes_{\alpha,r}^{\sigma} G$ is crossed product-like by $\{\lambda_{\sigma}(g)\}_{g\in G}$.
\end{exam}

\begin{exam}\upshape
Let $A\subseteq B$ be a crossed product-like inclusion of $\C^*$-algebras by $U$.
Then for a unital $\C^*$-algebra $C$, $A\otimes C\subseteq B\otimes C$ is a crossed product inclusion by $U\otimes I$.
\end{exam}

\begin{rem}\upshape
If $\mathbb{C}I \subseteq A$ is a crossed product-like inclusion of $\C^*$-algebras by a discrete amenable group,
then $A$ is strongly amenable.
Hence, the Cuntz algebras $\mathcal{O}_n$ are nuclear 
but $\mathbb{C}I\subseteq \mathcal{O}_n$ is not crossed product-like by discrete amenable groups.
\end{rem}

In the next lemma,  
to get a point-norm version of \cite[Lemma 3.3]{Chris2} 
we modify the argument of \cite[Lemma 3.2]{CSSWW} for crossed product-like inclusions by amenable groups.

\begin{lemma}\label{1.4}
Let $C\subseteq D$ be a unital inclusion of $\mathrm{C}^*$-algebras and 
let $A,B$ be intermediate $\mathrm{C}^*$-subalgebras for $C\subseteq D$ with a conditional expectation $E\colon D\to B$.
Suppose that $C\subseteq A$ is crossed product-like by a discrete amenable group $U$ and $d(A,B)<\gamma<1/4$.
Then
for any finite subset $X\subseteq A_1$ and $\varepsilon>0$,
there exists a unital $C$-fixed $(X,\varepsilon)$-approximate $*$-homomorphism $\phi\colon A\to B$ such that
\[
\| \phi -\id_A \| \le \left(8 \sqrt{2} + 2 \right)\gamma .
\]
\end{lemma}

\begin{proof}
Let a finite set $X\subseteq A_1$ and $0<\varepsilon<1$ be given.
Let $D$ act on a Hilbert space $H$.
By Stinespring's theorem, 
we can find a Hilbert space $K \supseteq H$ and a unital $*$-homomorphism $\pi\colon D\to \mathbb{B}(K)$ such that
\[
E(d)= P_H \pi(d) |_H, \ \ d\in D,
\]
because $E\colon D\to B$ is a unital cpc map.
Furthermore, $P_H\in \pi(B)'$, since $E$ is a $B$-fixed map.
By Lemma \ref{Arveson},
there exists $\eta>0$ such that
for any self-adjoint operator $t\in \mathbb{B}(K)$ with 
$\mathrm{s p}(t)\subseteq [0,2\gamma]\cup[1-2\gamma,1]$
and $x\in\mathbb{B}(K)$ with $\|x\|\le 2$,
the inequality $\| x t-t x\| <\eta$ implies $\| x p - p x\|  < \varepsilon^2/18$,
where $p$ is the spectral projection of $t$ for $[1-2\gamma,1]$.
There exist $\{ u_1,\dots,u_N\}\subseteq U$ and $\{c_i^{(x)} : 1\le i \le N, x\in X\}\subseteq C_1$ such that
\[
\left\| x - \sum_{i=1}^{N} c_i^{(x)} u_i \right\| <\frac{\varepsilon}{3}, \ \ x\in X.
\]
Let $\tilde{x}:=\sum_{i=1}^N c_i^{(x)} u_i$ for $x\in X$.
Then $\|\tilde{x} \| \le \|x\|+\varepsilon<2$.
Since $U$ is amenable, we may choose a finite subset $F\subseteq U$ satisfying
\[
\frac{| u_i F\bigtriangleup F  |}{|F|} <\frac{\eta}{N}, \ \ 1\le i \le N.
\]
Define
\[
t := \frac{1}{|F|}\sum_{v\in F} \pi(v) P_H \pi(v^*) \in \mathbb{B}(K).
\]
Since $U\subseteq \mathcal{N}_A(C)$ and $P_H\in \pi(C)'$, we have $t \in \pi(C)'$.
For any $x\in X$,
\begin{align*}
\pi(\tilde{x}) t
&=\sum_{i=1}^{N} \pi( c_i^{(x)} u_i ) \frac{1}{|F|} \sum_{v\in F} \pi(v) P_H \pi(v^*) \\
&=\frac{1}{|F|} \sum_{i=1}^{N} \sum_{v\in F}   \pi( c_i^{(x)}u_i v) P_H \pi (v^* ) \\
&=\frac{1}{|F|} \sum_{i=1}^{N} \sum_{\tilde{v}\in u_i F}   \pi( c_i^{(x)}\tilde{v}) P_H \pi (\tilde{v}^*u_i )
\end{align*}
and
\begin{align*}
t \pi(\tilde{x})
&= \frac{1}{|F|} \sum_{v\in F} \pi(v) P_H \pi(v^*) \sum_{i=1}^{N} \pi( c_i^{(x)} u_i ) \\
&= \frac{1}{|F|} \sum_{i=1}^{N} \sum_{v\in F} \pi( c_i^{(x)} v) P_H \pi ( v^* u_i ) .
\end{align*}
Therefore, 
\begin{equation}\label{eta}
\left\|  \pi(\tilde{x}) t - t  \pi(\tilde{x}) \right\| 
\le \sum_{i=1}^{N} \frac{| u_i F\bigtriangleup F|}{|F|} <\eta, \ \ x \in X.
\end{equation}
For $v\in F$, there exists $v'\in B_1$ such that 
$\| v - v' \| <\gamma$.
Since $P_H\in \pi(B)'$, we have 
\[
\| \pi(v) P_H - P_H \pi(v) \| \le 
\| \pi(v) P_H - \pi(v') P_H\| + \| P_H \pi(v') - P_H \pi (v) \| \le 2 \gamma.
\]
Thus, $\mathrm{s p}(t) \subseteq [0,2\gamma]\cup [1-2\gamma,1]$, since
\begin{align*}
\| t -P_H\| 
&=\left\| \frac{1}{|F|}\sum_{v\in F} \pi(v) P_H \pi(v^*) -  \frac{1}{|F|}\sum_{v\in F} P_H \pi(v) \pi(v^*) \right\| \\
&\le \frac{1}{|F|} \sum_{v\in F} \| \pi(v) P_H - P_H \pi(v) \| \| \pi(v^*)\| 
\le 2 \gamma.
\end{align*}
Let $q=\chi_{[1-2\gamma,1]}(t) \in C^* ( t, I_K)$.
By (\ref{eta}), 
\begin{equation}\label{ab}
\| \pi(\tilde{x}) q- q\pi(\tilde{x})\| < \frac{\varepsilon^2}{18}, \ \ x\in X.
\end{equation}
Since $\| q- P_H\| \le 2\| t- P_H\| <1$, 
there exists a unitary $w\in C^*( t, P_H , I_K)$ such that 
$w P_H w^* =q$ and $\| w-I_K\| \le  \sqrt{2} \| q- P_H\|$.
Define $\phi\colon A\to \mathbb{B}(K)$ by
\[
\phi(a)= P_H w^* \pi(a) w |_H, \ \ a\in A.
\]
Since $w\in C^*(t, P_H, I_K) \subseteq C^*( \pi(A) ,P_H)$ and 
$P_H \pi(A)|_H = \mathrm{ran}(E) \subset B$,
the range of $\phi$ is contained in $B$.
Furthermore, $\phi|_C=\id_C$ because $w\in C^*(t,P_H,I_K) \subseteq \pi(C)'$.

For $x \in X\cup X^*$,
by using $P_H w^* =P_H w^* q$ and (\ref{ab}),
\begin{align}
\begin{split}
\label{cc}
\| \phi(\tilde{x} \tilde{x}^*) -\phi(\tilde{x})\phi(\tilde{x}^*)\| 
&=\| P_H w^* \pi(\tilde{x}\tilde{x}^*) w P_H - P_H w^* \pi(\tilde{x}) w P_H w^* \pi(\tilde{x}^*)w P_H \| \\
&=\| P_H w^* q \pi(\tilde{x} \tilde{x}^*) w P_H - P_H w^* \pi(\tilde{x}) q \pi(\tilde{x}^*) w P_H \| \\
&\le \| q \pi(\tilde{x}) -\pi (\tilde{x}) q \| \| \pi (\tilde{x}^*) \|   < \frac{\varepsilon^2}{9} .
\end{split}
\end{align}
Therefore,
by (\ref{cc}) and Proposition \ref{1/2},
\begin{align*} 
&\| \phi(x x^*) -\phi(x)\phi(x^*)\|  \\
&\le \| \phi(x x^*) - \phi(x \tilde{x}^*) \| + 
\| \phi(x \tilde{x}^*) -\phi(x) \phi(\tilde{x}^*)\| +
\| \phi(x)\phi(\tilde{x})-\phi(x)\phi(x^*)\| \\
&\le \| x x^* - x \tilde{x}^*\| +
\| \phi( \tilde{x}\tilde{x}^*)-\phi(\tilde{x})\phi(\tilde{x}^*)\|^{1/2}\| x\| + 
\|\phi(x)\| \| \phi(\tilde{x}^*)-\phi(x^*)\| \\
&\le \frac{\varepsilon}{3} + \frac{\varepsilon}{3} + \frac{\varepsilon}{3} = \varepsilon.
\end{align*}

For $a\in A_1$, we have
\begin{align*}
\| \phi(a) - a \| 
&\le \| \phi(a) - E(a) \| + \| E(a) -a \| \\
&\le \| P_H w^* \pi(a) w P_H - P_H \pi(a) P_H\| + 2 d(A,B) \\
&\le 2 \|w-I_K\| +2d(A,B) \le (8\sqrt{2}+2)\gamma,
\end{align*}
and the lemma follows.
\end{proof}

The next lemma is a version of \cite[Lemma 3.4]{CSSWW} for crossed product-like inclusions by amenable groups.

\begin{lemma}\label{1.5}
Let $A,B$ and $C$  be $\mathrm{C}^*$-algebras with a common unit.
Suppose that $C\subseteq A\cap B$ and $C\subseteq A$ is crossed product-like by a discrete amenable group $U$.
Then for any finite set $X\subseteq A_1$ and $\varepsilon>0$, 
there exist a finite set $Y\subseteq A_1$ and $\delta>0$ with the following property\,$:$
Given $\gamma<1/10$ and 
two unital $C$-fixed $(Y,\delta)$-approximation $*$-homomorphisms 
$\phi_1,\phi_2\colon A\to B$ with $\phi_1\approx_{Y,\gamma} \phi_2$,
there exists a unitary $u\in C'\cap B$ 
such that 
\[
\phi_1 \approx_{X,\varepsilon} \mathrm{Ad}(u) \circ \phi_2  \ \ \text{and} \ \  
\| u-I\| \le \sqrt{2}(\gamma +\delta).
\]
\end{lemma}

\begin{proof}
Let a finite set $X\subseteq A_1$ and $0<\varepsilon<1$ be given.
There exist $\{ u_1,\dots,u_N \} \subseteq U$ and 
$\{ c_i^{(x)} : 1\le i \le N, x\in X \} \subseteq C_1$
such that
\[
\left\| x- \sum_{i=1}^N c_i^{(x)} u_i \right\| <\frac{\varepsilon}{3} , \ \ x\in X.
\]
Let $\tilde{x}:=\sum_{i=1}^{N} c_i^{(x)} u_i^{(x)}$ for $x\in X$.
Then $\| \tilde{x} \| \le 1+ \varepsilon/3< 2$.
By Lemma \ref{Arveson}, 
there exists $\eta>0$ such that for any $s\in B_1$ and $a\in B$ with $\| a\|\le 2$, 
the inequality $\| s s^* a- a s s^*\| <\eta$ implies $\| |s| a- a |s| \| < \varepsilon/12$.
Let 
\[
0<\delta < \min \left\{ \left(\frac{\varepsilon}{60}\right)^2, \frac{\eta^2}{100} \right\}.
\]
There exists a finite set $Y \subseteq U$ such that
\[
\frac{ | u Y \bigtriangleup Y | }{ | Y | } <\frac{\delta}{N}, \ \ 
u\in \left\{ u_i, u_i^* :  1\le i\le N \right\}.
\]

Let $\gamma<1/10$ and 
$\phi_1, \phi_2\colon A\to B$ be $C$-fixed $(Y,\delta)$-approximation $*$-homomorphisms 
with $\phi_1 \approx_{Y,\gamma} \phi_2$.
Define
\[
s:= \frac{1}{ | Y | } \sum_{v\in Y} \phi_1(v) \phi_2(v^*).
\]
Since $\phi_1$ and $\phi_2$ are $C$-fixed maps and for $u\in U$, $u C u^*=C$, 
we have $s\in C'\cap B$.
By Proposition \ref{1/2}, for $x\in X$ and $v\in Y$,
\begin{align}
&\| \phi_1 (\tilde{x} v)-\phi_1 (\tilde{x}) \phi_1 (v) \| \le \| \phi_1 (v v^* ) - \phi_1 (v) \phi_1 (v^*) \|^{1/2} \|\tilde{x} \| \le 2 \sqrt{\delta}, \label{ad} \\
&\| \phi_2( v^* \tilde{x} ) -\phi_2(v^*) \phi_2( \tilde{x})\| \le \| \phi_2( v^* v)-\phi_2(v^*)\phi_2(v)\|^{1/2}\| \tilde{x} \| \le 2 \sqrt{\delta}. \label{ae}
\end{align}
Furthermore,
\begin{align}
\begin{split}\label{aa}
\frac{1}{|Y|} \sum_{v\in Y} \phi_1(\tilde{x} v) \phi_2(v^*) 
&= \frac{1}{|Y|} \sum_{v\in Y} \sum_{i=1}^{N} \phi_1 \left( c_i^{(x)} u_i  v \right) \phi_2(v^*) \\
&= \frac{1}{|Y|} \sum_{i=1}^{N} \sum_{v\in u_i Y} \phi_1 \left( c_i^{(x)} v \right) \phi_2 \left( v^* u_i \right)
\end{split}
\end{align}
and
\begin{align}
\begin{split}\label{ab}
\frac{1}{|Y|} \sum_{v\in Y} \phi_1( v) \phi_2(v^* \tilde{x} ) 
&=\frac{1}{|Y|} \sum_{v\in Y} \sum_{i=1}^{N} \phi_1(v) \phi_2 \left( v^* c_i^{(x)} u_i \right) \\
&=\frac{1}{|Y|} \sum_{i=1}^{N} \sum_{v\in Y} \phi_1 \left( c_i^{(x)} v \right) \phi_2 \left( v^* u_i \right)
\end{split}
\end{align}
By (\ref{aa}), (\ref{ab}) and the choice of $Y$, for $x\in X$,
\begin{align}
\begin{split}\label{ac}
\left\| \frac{1}{|Y|} \sum_{v\in Y} \phi_1(\tilde{x} v) \phi_2(v^*)  - \frac{1}{|Y|} \sum_{v\in Y} \phi_1( v) \phi_2(v^* \tilde{x} ) \right\| < \sum_{i=1}^N \frac{ | u_i^{(x)} Y \bigtriangleup Y | }{ | Y | } <\delta.
\end{split}
\end{align}
Similarly, we have
\begin{equation}\label{az}
\left\| \frac{1}{|Y|} \sum_{v\in Y} \phi_1(\tilde{x}^* v) \phi_2(v^*)  - \frac{1}{|Y|} \sum_{v\in Y} \phi_1( v) \phi_2(v^* \tilde{x}^* ) \right\|
 <\delta, \ \ x\in X.
\end{equation}
By (\ref{ad}), (\ref{ae}) and (\ref{ac}), 
\begin{equation}\label{af}
\| \phi_1( \tilde{x} ) s- s \phi_2(\tilde{x}) \| \le \delta + 4 \sqrt{\delta} < 5 \sqrt{\delta}, \ \ x \in X \cup X^*.
\end{equation}
By taking adjoints,
\[
\| s^* \phi_1(\tilde{x}) - \phi_2(\tilde{x}) s^* \| \le 5 \sqrt{\delta} , \ \ x \in X\cup X^*.
\]
Thus, for $x\in X\cup X^*$,
\begin{align*}
\| \phi_2(\tilde{x}) s^* s  - s^* s  \phi_2(\tilde{x}) \| 
&\le \| \phi_2(\tilde{x}) s^*s - s^* \phi_1( \tilde{x}) s \| + \| s^* \phi_1(\tilde{x} ) s - s^* s \phi_2(\tilde{x}) \| \\
&\le \| \phi_2(\tilde{x}) s^* - s^* \phi_1(\tilde{x}) \| \| s\|  + \| s^* \| \| \phi_1(\tilde{x}) s -s \phi_2(\tilde{x}) \| \\
&\le 10 \sqrt{\delta}<\eta.
\end{align*}
By the choice of $\eta$ and (\ref{ah}),
\begin{equation}\label{ah}
\| \phi_2(\tilde{x}) |s| - |s| \phi_2(\tilde{x}) \| < \frac{\varepsilon}{12}, \ \ x\in X\cup X^*.
\end{equation}
Since $\phi_1$ is a $(Y,\delta)$-approximation $*$-homomorphism and $\phi_1\approx_{Y,\gamma} \phi_2$, we have
\begin{align}
\begin{split}\label{aj}
\| s -I \| 
&= \left\| \frac{1}{ |Y| } \sum_{v\in Y} \phi_1(v) \phi_2(v^*) - \frac{1}{ |Y| } \sum_{v\in Y} \phi_1(v v^*) \right\| \\
&\le \frac{1}{ |Y| } \sum_{v\in Y} \left\|  \phi_1(v) \phi_2(v^*) -  \phi_1(v) \phi_1(v^*) \right\| 
+  \frac{1}{ |Y| } \sum_{v\in Y} \left\| \phi_1(v) \phi_1(v^*)  - \phi_1(v v^*) \right\| \\
&\le \gamma + \delta<1.
\end{split}
\end{align}
Since this inequality gives invertibility of $s$, 
the unitary $u$ in the polar decomposition $s=u|s|$ lies in 
$C^*(s, I)\subseteq C'\cap B$ and satisfies $\| u- I \| \le \sqrt{2}(\gamma+\delta)$.
Then, by (\ref{aj}),
\begin{align*}
\| |s | -I \| 
\le \| u^*s- I \| 
\le \| s- I\| + \| I -u\| 
\le (1+ \sqrt{2})(\gamma+ \delta)< \frac{1}{2}.
\end{align*}
Hence, $\| |s|^{-1} \| \le 2$ so,
\begin{align}\label{ak}
\begin{split}
\| \phi_1(\tilde{x}) - u \phi_2(\tilde{x})u^*  \|
&= \| \phi_1(\tilde{x})u - u \phi_2(\tilde{x}) \| \\
&\le \| \phi_1(\tilde{x}) u|s| - u \phi_2(\tilde{x}) |s| \| \, \| |s|^{-1}\| \\
&\le 2 \| \phi_1(\tilde{x}) u|s| - u \phi_2(\tilde{x}) |s| \| \\
&\le 2 \| \phi_1(\tilde{x}) s- s \phi_2(\tilde{x})\| +2 \| s \phi_2(\tilde{x}) - u \phi_2(\tilde{x}) |s| \| \\
&\le 10 \sqrt{\delta}+ 2\| |s| \phi_2(\tilde{x}) - \phi_2(\tilde{x}) |s| \| \\
&\le 10 \sqrt{\delta} + \frac{\varepsilon}{6} < \frac{\varepsilon}{3},
\end{split}
\end{align}
for $x\in X$, by (\ref{af}), (\ref{ah}) and (\ref{aj}).
For $x\in X$,
by (\ref{ak}),
\begin{align*}
\| \phi_1(x) - u \phi_2(x) u^* \| 
&\le \| \phi_1(x) - \phi_1(\tilde{x}) \| + \| \phi_1(\tilde{x}) - u \phi_2(\tilde{x}) u^* \| + \| u \phi_2(\tilde{x}) u^* -u \phi_2(x) u^* \| \\
&< \frac{\varepsilon}{3} + \frac{\varepsilon}{3} + \frac{\varepsilon}{3} = \varepsilon.
\end{align*}
Therefore, $\phi_1 \approx_{X, \varepsilon} \mathrm{Ad}(u) \circ \phi_2$.
\end{proof}

\begin{rem}\upshape
Let a pair $(Y,\delta)$ hold Lemma \ref{1.5}.
Then for any finite set $Y'\supseteq Y$ and constant $\delta'<\delta$,
a pair $(Y',\delta')$ holds Lemma \ref{1.5}.
\end{rem}

By Lemma \ref{1.4} and \ref{1.5},
we can show Theorem A.

\begin{thm}\label{irreducible}
Let $C\subseteq D$ be a unital irreducible inclusion of $\mathrm{C}^*$-algebras 
acting on a separable Hilbert space $H$.
Let $A$ and $B$ be separable intermediate $\C^*$-subalgebras for $C\subseteq D$ with 
a conditional expectation $E\colon D\to B$.
Suppose that 
$C\subseteq A$ is crossed product-like by a discrete amenable group.
If $\dab<140^{-1}$,
then $A= B$.
\end{thm}

\begin{proof}
Let $a\in A_1,\varepsilon>0$ and $\dab<\gamma<140^{-1}$ be given.
By Lemma \ref{1.5},
there exist a finite subset $Y\subseteq A_1$ and $\delta>0$ with the following property:
Given $\gamma'<1/10$ and two unital $C$-fixed $(Y,\delta)$-approximate $*$-homomorphisms $\phi_1,\phi_2\colon A\to D$
with $\phi_1\approx_{Y,\gamma'}\phi_2$,
there exists a unitary $u\in C'\cap D$ such that 
\[
\| \phi_1(a) - (\Ad(u) \circ \phi_2)(a) \| \le \varepsilon.
\]
By Lemma \ref{1.4},
there exists a unital $C$-fixed $(Y,\delta)$-approximate $*$-homomorphism $\phi\colon A\to B$ such that 
$\| \phi - \id_A\| \le \left(8 \sqrt{2} + 2 \right)\gamma$.
Then there exists a unitary $u\in C'\cap D$ such that 
\[
\| \phi(a) - (\Ad(u)) (a) \| \le \varepsilon
\]
by the definition of $Y$ and $\delta$. 
Since $u\in C'\cap D=\mathbb{C}I$, we have $\| \phi(a) - a\| \le \varepsilon$.
Therefore, since $\phi(a)\in B$ and $\varepsilon$ is arbitrary,
$a\in B$, that is, $A\subseteq B$.
Furthermore, by Lemma \ref{surjective},
the theorem follows.
\end{proof}

In the following lemma, we modify \cite[Lemma 3.6]{CSSWW}, which is a Kaplansky density style result for approximate commutants.

\begin{lemma}\label{1.6}
Let $C\subseteq A$ be a unital inclusion of non-degenerate $\mathrm{C}^*$-algebras in $\mathbb{B}(H)$.
Suppose that $C\subseteq A$ is crossed product-like by a discrete amenable group $U$.
Then
for any finite set $X\subseteq A_1$ and $\varepsilon, \mu >0$,
there exist a finite set $Y\subseteq A_1$ and $\delta>0$
with the following property$\colon$
Given a finite set $S\subseteq H_1$ and a self-adjoint operator $m\in \overline{C'\cap A_1}^{\w}$ with 
\begin{equation}
\| m y - y m\| \le \delta, \ \ y\in Y,
\end{equation}
there exists a self-adjoint operator $a\in C'\cap A_1$ such that $\| a\| \le \| m\|$,
\begin{align}
\| ax- x a\| <\varepsilon, \ \ x\in X, 
\end{align}
and
\begin{align}
\| (a-m)\xi\| <\mu \ \ and \ \ \| (a-m)^* \xi\| <\mu , \  \ \xi \in S.
\end{align}
\end{lemma}

\begin{proof}
Let a finite set $X\subseteq A_1$ and $\varepsilon,\mu>0$ be given.
There exist $\{u_1,\dots,u_{N}\}$\hspace{0pt}$\subseteq U$ and 
$\{ c_i^{(x)}: 1\le i\le N, x\in X\}\subseteq C_1$ such that
\[
\left\| x- \sum_{i=1}^{N} c_i^{(x)} u_i \right\| <\frac{\varepsilon}{3}, \ \ x\in X.
\]
Let $\tilde{x}:=  \sum_{i=1}^{N} c_i^{(x)} u_i$ for $x\in X$.
Since $U$ is amenable, there exists a finite set $F\subseteq U$ such that
\begin{align}\label{1.6.1}
\frac{ | u_i F\bigtriangleup F | }{|F|} <\frac{\varepsilon}{3N}, \ \ 1\le i\le N.
\end{align}
Define $Y:=F\cup F^*$ and $\delta:=\mu/2$.

Let $S$ be a finite set in $H_1$ and 
$m \in \overline{C'\cap A_1}^{\w}$ be a self-adjoint operator with 
\begin{align}\label{1.6.2}
\| m y - y m \| < \delta, \ \ y\in Y.
\end{align}
By Kaplansky's density theorem, 
there exists a self-adjoint operator $a_0\in C'\cap A_1$ such that $\| a_0\| \le \| m\|$,
\begin{align}\label{1.6.3}
\| (a_0-m) v^*\xi\| <\mu \ \ \text{and} \ \ \| (a_0-m)^* v \xi\| <\mu , \ \ v\in Y, \ \xi\in S.
\end{align}
Define
\[
a:= \frac{1}{|F|}\sum_{v\in F} v a_0 v^*.
\] 
Then, $\| a \| \le \| a_0 \| \le \|m \|$.

For any $x\in X$,
\begin{align}\label{1.6.4}
\begin{split}
\| \tilde{x}a- a\tilde{x}\| 
&= \left\| \frac{1}{|F|} \sum_{i=1}^{N}\sum_{v\in F} c_i^{(x)} u_i v a_0v^*  - 
\frac{1}{|F|} \sum_{i=1}^{N} \sum_{v\in F} c_i^{(x)} v a_0 v^* u_i  \right\| \\
&= \left\| \frac{1}{|F|} \sum_{i=1}^{N} c_i^{(x)} \left(
\sum_{\tilde{v}\in u_i F}  \tilde{v} a_0 \tilde{v}^* u_i
- \sum_{v\in F} v a_0 v^* u_i  \right) \right\| \\
&\le \sum_{i=1}^{N} \frac{| u_i F \bigtriangleup F|}{|F|} <\frac{\varepsilon}{3},
\end{split}
\end{align}
by (\ref{1.6.1}).
For $x\in X$,
since $\| x- \tilde{x}\| <\varepsilon/3$,
\begin{align*}
\| x a - a x\|
&\le \| x a- \tilde{x} a\| + \| \tilde{x}a- a \tilde{x}\| + \| \tilde{x}a- x a\| \\
&\le \frac{\varepsilon}{3} + \frac{\varepsilon}{3}+ \frac{\varepsilon}{3}=\varepsilon
\end{align*}
by (\ref{1.6.4}).

For $\xi \in S$, by (\ref{1.6.2}) and (\ref{1.6.3}),
\begin{align*}
&\| (a-m) \xi\| \\
&\le \left\| \left( \frac{1}{|F|}\sum_{v\in F} v a_0 v^* - \frac{1}{|F|}\sum_{v\in F} v m v^* \right) \xi \right\| 
+ \left\| \left( \frac{1}{|F|}\sum_{v\in F} v m v^* - \frac{1}{|F|}\sum_{v\in F} v v^* m \right) \xi \right\| \\
&\le \max_{v\in F} \|(a_0-m) v^*\xi \| + \max_{v\in F}\| m v^* - v^* m \| \\
&\le \frac{\mu}{2}+ \delta = \mu.
\end{align*}

Similarly, for $\xi\in S$,
\begin{align*}
\| (a-m)^*\xi\|
&\le \max_{v\in F} \| (a_0 -m )^* v \xi\| + \max_{v\in F} \|  m v - v m\| \le \frac{\mu}{2}+ \delta =\mu,
\end{align*}
and the lemma follows.
\end{proof}

By Lemma \ref{1.7} and \ref{1.6}, we obtain the following version of Lemma \ref{1.6} for unitary operators.
We need the next lemma in Section \ref{unitary}.

\begin{lemma}\label{1.8}
Let $C\subseteq A$ be a unital inclusion of non-degenerate $\mathrm{C}^*$-algebras in $\mathbb{B}(H)$.
Suppose that $C\subseteq A$ is crossed product-like by a discrete amenable group $U$.
Then
for any finite set $X\subseteq A_1$, $\varepsilon_0, \mu_0 >0$ and $0<\alpha<2$,
there exist a finite set $Y\subseteq A_1$ and $\delta_0>0$
with the following property$\colon$Given a finite set $S\subseteq H_1$
 and a unitary $u \in \overline{C'\cap A}^{\w}$ with $\| u - I_H\| \le \alpha$ and 
\begin{equation}
\| u y - y u \| \le \delta_0 , \ \ y\in Y,
\end{equation}
there exists a unitary $v\in C'\cap A$ such that $\|  v - I_H \| \le \| u - I_H \|$,
\begin{align}
\| v x- x v \| <\varepsilon_0, \ \ x\in X, 
\end{align}
and
\begin{align}
\| (v-u)\xi\| <\mu_0 \ \ and \ \ \| (v-u)^* \xi\| <\mu_0 , \  \ \xi \in S.
\end{align}
\end{lemma}

\begin{proof}
Let a finite set $X\subseteq A_1$, $\varepsilon_0, \mu_0 >0$ and $0<\alpha<2$ be given.
There exists $0<c <\pi$ such that $| 1 - e^{i \pi \theta}| \le \alpha$
if and only if $\theta \in [ -c , c]$ modulo $2\pi$.
By Lemma \ref{Arveson},
there exists $\varepsilon>0$ such that 
given a self-adjoint operator $k \in \mathbb{B}(H)_1$ and $a \in \mathbb{B}(H)_1$,
if $\| a k - k a \| <\varepsilon$, then $\| a e^{i\pi k} - e^{i\pi k} a\| <\varepsilon_0$.

By Lemma \ref{1.7},
there exists $\mu>0$ with the following property:
Given a finite set $S\subseteq H_1$ and a self-adjoint operator $k\in \mathbb{B}(H)_1$,
there exists a finite set $S'\subseteq H_1$
such that for a self-adjoint operator $k\in \mathbb{B}(H)_1$,
if 
\begin{align*}
\| (h-k)\xi'\| <\mu_0, \ \ \xi'\in S',
\end{align*}
then
\begin{align*}
\| (e^{i\pi h} - e^{i\pi k}) \xi\| <\mu \ \ \text{and} \ \ \| (e^{i\pi h} - e^{i\pi k})^* \xi\| <\mu, \ \ \xi \in S.
\end{align*}
By Lemma \ref{1.6},
there exist a finite set $Y\subseteq A_1$ and $\delta>0$
with the following property:
For any finite set $S\subseteq H_1$ and self-adjoint operator $m\in \overline{C'\cap A_1}^{\w}$ with
\begin{align*}
\| m y - y m \| < \delta, \ \ y\in Y,
\end{align*}
there exists a self-adjoint operator $a \in C'\cap A_1$ such that $\| a \| \le \| m \|$,
\begin{align*}
\| a x &- x a\| <\varepsilon, \ \ x\in X, \\
\| (a - m) \xi \| <\mu& \ \ \text{and} \ \ \| (a-m)^* \xi\| <\mu , \ \ \xi \in S.
\end{align*}
By Lemma \ref{Arveson},
there exists $\delta_0>0$ such that
for any  $y \in \mathbb{B}(H)$ and unitary $u \in \mathbb{B}(H)$ with $\| u - I_H\| \le \alpha$,
if  $\| u y - y u\| \le \delta_0$, then
\begin{align*}
\left\| \frac{\log u}{\pi} y - y \frac{\log u}{\pi} \right\| \le \delta.
\end{align*}

Given a finite set $S\subset H_1$ and  a unitary $ u\in \overline{C'\cap A}^{\w}$ with $\| u - I_H\| \le \alpha$ and
\begin{align*}
\| u y - y u \| \le \delta_0, \ \ y\in Y.
\end{align*}
Let
\[
h:= -i \frac{\log u}{\pi} \in C'\cap M.
\]
By the definition of $\delta_0$,
\begin{align*}
\|  h y - y h \|  \le \delta, \ \ y\in Y.
\end{align*}
By the definition of $\mu$,
there exists a finite set $S'\subseteq H_1$ such that
for any self-adjoint operator $k \in \mathbb{B}(H)_1$, 
if 
\begin{align*}
\| (h - k) \xi' \| <\mu_0, \ \ \xi'\in S',
\end{align*}
then
\begin{align*}
\| (e^{i\pi h}- e^{i\pi k})\xi \| <\mu \ \ \text{and} \ \ \| (e^{i\pi h} - e^{i\pi k})^* \xi\| <\mu, \ \ \xi\in S.
\end{align*}
By the definitions of $Y$ and $\delta$,
there exists a self-adjoint operator $k\in C'\cap A_1$
such that $\|k\| \le \|h\|$,
\begin{align*}
\| k x &- x k \| <\varepsilon, \ \ x\in X, \\
\| (h - k) \xi' \| <\mu& \ \ \text{and} \ \ \| (h - k)^* \xi' \| <\mu , \ \ \xi' \in S'.
\end{align*}
Define $v:= e^{i\pi k}$.
Then, we have $\| v- I_H\| \le \| e^{i\pi h} - I_H \| = \| u- I_H\| $.

By the definition of $\varepsilon$ and $S'$,
we have
\[
\| v x - x v \| <\varepsilon_0 , \ \ x\in X
\]
and
\[
\| (v - u) \xi \| <\mu_0 \ \ \text{and} \ \ \| (v - u )^* \xi \| <\mu_0 , \ \ \xi \in S.
\]
Hence the lemma is proved.
\end{proof}


\section{Isomorphisms}\label{isomorphism}


In this section,
we show Theorem B.
Given a unital inclusion $C\subseteq D$ of C$^*$-algebras and 
intermediate C$^*$-subalgebras $A,B$ for this inclusion 
with a conditional expectation form $D$ onto $B$,
if $A=C\rtimes G$, where $G$ is a discrete amenable group, 
and if $A$ and $B$ are sufficiently close,
then $A$ must be $*$-isomorphic to $B$.
To do this, we modify \cite[Lemma 4.1]{CSSWW} in the next lemma.
The approximation approach of \cite[Lemma 4.1]{CSSWW} inspired by the intertwining arguments of \cite[Theorem 6.1]{Chris5}.

\begin{lemma}\label{3.1}
Let $C\subseteq D$ be a unital inclusion of $\C^*$-algebras and 
let $A,B$ be separable intermediate $\mathrm{C}^*$-subalgebras for $C\subseteq D$ with a conditional expectation $E \colon D\to B$.
Let $\{a_n\}_{n=0}^{\infty}$ be a dense subset in $A_1$ with $a_0=0$.
Suppose that $C\subseteq A$ is crossed product-like by a discrete amenable group
and $d(A,B)<\gamma<10^{-3}$.
Put $\eta:=(8\sqrt{2}+2)\gamma$.
Then
for any finite set $X \subseteq A_1$,
there exist finite subsets $\{X_n\}_{n=0}^{\infty}, \{ Y_n\}_{n=0}^{\infty}\subseteq A_1$, 
positive constants $\{ \delta_n \}_{n=0}^{\infty}$, 
$C$-fixed cpc maps $\{ \theta_n\colon A\to B\}_{n=0}^{\infty}$ and
unitaries $\{u_n\}_{n=1}^{\infty}\subseteq C'\cap B$ with the following conditions$\colon$
\begin{enumerate}
\item[\upshape{(a)}] 
For $n\ge 0$, $\delta_n < \min\{ 2^{-n}, \gamma \}$, $a_n\in X_n\subseteq X_{n+1}$ and $X \subseteq X_1;$

\item[\upshape{(b)}] 
For $n\ge0$ and two unital $C$-fixed $(Y_n,\delta_n)$-approximation $*$-\hspace{0pt}homomorphisms $\phi_1,\phi_2 \colon A\to B$ with 
$\phi_1\approx_{Y_n,2\eta}\phi_2$,
there exists a unitary $u\in C'\cap B$ such that $\mathrm{Ad}(u)\circ \phi_1\approx_{X_n,\gamma/2^n}\phi_2$ and 
$\| u- I\| \le \sqrt{2}(2\eta+\delta_n);$

\item[\upshape{(c)}]
For $n\ge 0$, $X_n\subseteq Y_n;$

\item[\upshape{(d)}]
For $n\ge 0$, $\theta_n$ is a $(Y_n,\delta_n)$-approximation $*$-homomorphism with $\| \theta_n-\id_A\| \le \eta;$

\item[\upshape{(e)}]
For $n\ge 1$, $\mathrm{Ad}(u_n)\circ \theta_n\approx_{X_{n-1},\gamma/2^{n-1}}\theta_{n-1}$ and $\| u_n- I\| \le \sqrt{2}(2\eta+\delta_{n-1})$.
\end{enumerate}
\end{lemma}

\begin{proof}
We prove this lemma by the induction.
Let a finite subset $X\subseteq A_1$ be given.
Let $X_0=\{ 0 \}=\{ a_0\}=Y_0$, $\delta=1$ and $\theta:=E|_A \colon A\to B$.

Suppose that we can construct completely up to the $n$-th stage.
We will write the condition (a) for $n$ as (a)$_n$.
Let $X_{n+1}:= X_n\cup X \cup \{a_{n+1} \} \cup Y_n$.
By Lemma \ref{1.5}, 
there exist a finite set $Y_{n+1}\subseteq A_1$ and 
$0<\delta_{n+1} < \min\{ \delta_n, 2^{-(n+1)}, \gamma \}$
satisfying condition (b)$_{n+1}$ and $X_{n+1}\subseteq Y_{n+1}$.
By Lemma \ref{1.4},
there exists a unital $C$-fixed $(Y_{n+1},\delta_{n+1})$-approximation $*$-\hspace{0pt}homomorphism
$\theta_{n+1}\colon A\to B$ such that $\| \theta_{n+1}- \id_A\| \le \eta$.
Therefore, $X_{n+1}, Y_{n+1}, \delta_{n+1}$ and $\theta_{n+1}$ satisfy (a)$_{n+1}$, (b)$_{n+1}$, (c)$_{n+1}$ and (d)$_{n+1}$.

Since $Y_n\subseteq Y_{n+1}$ and $\delta_{n+1}<\delta_n$,
$\theta_n$ and $\theta_{n+1}$ are unital $C$-fixed $(Y_n,\delta_n)$-approximation $*$-homomorphisms with $\| \theta_n- \theta_{n+1} \| \le 2\eta$.
Thus, by (b)$_n$, there exists a unitary $u_{n+1}\in C'\cap B$ such that 
$\mathrm{Ad}(u_{n+1})\circ \theta_{n+1}\approx_{X_n, \gamma/2^n}\theta_n$ and 
$\| u_{n+1}- I\|\le \sqrt{2}(2\eta +\delta_n)$.
Then (e)$_{n+1}$ follows.
\end{proof}

\begin{prop}\label{3.2}
Let $C\subseteq D$ be a unital inclusion of $\mathrm{C}^*$-algebras and 
let $A$ and $B$ be separable intermediate $\mathrm{C}^*$-subalgebras for $C\subseteq D$ with a conditional expectation $E\colon D\to B$.
Suppose that $C\subseteq A$ is crossed product-like by a discrete amenable group
and $d(A,B)<\gamma<10^{-3}$.
Then
for any finite subset $X\subseteq A_1$,
there exists a $C$-fixed surjective $*$-isomorphism $\alpha\colon A\to B$ such that 
\[
\alpha\approx_{X, 15\gamma} \id_A.
\]
\end{prop}

\begin{proof}
Let $\{a_n\}_{n=0}^{\infty}$ be a dense subset in $A_1$ with $a_0=0$.
Put $\eta:=(8\sqrt{2}+2)\gamma$.
By Lemma \ref{3.1},
we can construct $\{X_n\}_{n=0}^{\infty}, \{Y_n\}_{n=0}^{\infty}\subseteq A_1$, 
$\{\delta_n \}_{n=0}^{\infty}$,
$\{ \theta_n\colon A\to B\}_{n=0}^{\infty}$ and $\{u_n\}_{n=1}^{\infty}\subseteq C'\cap B$ which satisfy conditions (a)-(e) of that lemma.
For any $n\ge 1$, define
\[
\alpha_n:= \mathrm{Ad}(u_1\cdots u_n)\circ \theta_n.
\]
Fix $k\in \mathbb{N}$ and $x\in X_k$.
For any $n\ge k$,
\begin{align}\label{3.2.1}
\begin{split}
\| \alpha_{n+1}(x)-\alpha_n(x) \| 
&=\| \left( \mathrm{Ad}(u_1\cdots u_n)\circ \mathrm{Ad} (u_{n+1}) \circ \theta_{n+1} - \mathrm{Ad}(u_1\cdots u_n) \circ \theta_n \right)(x) \| \\
&= \| \left( \mathrm{Ad}(u_{n+1}) \circ \theta_{n+1} - \theta_n \right) (x) \| 
\le \frac{\gamma}{2^n}
\end{split}
\end{align}
For any $\varepsilon>0$, there exists $N\ge k$ such that $\gamma/2^{N-1} < \varepsilon$.
For any two natural numbers $m \ge n\ge N$, by (\ref{3.2.1}),
\[
\| \alpha_m(x) - \alpha_n(x) \| 
\le \sum_{j=n}^{m-1} \| \alpha_{j+1}(x) -\alpha_j(x) \| \le \sum_{j=n}^{m-1} \frac{\gamma}{2^j} <\frac{\gamma}{2^{N-1}} <\varepsilon.
\]
Thus,
$\{ \alpha_n(x) \}$ is a cauchy sequence.
Since $\bigcup_{n=0}^{\infty} X_n$ is dense in $A_1$, 
the sequence $\{ \alpha_n \}$ converges in the point-norm topology to a $C$-fixed cpc map $\alpha\colon A\to B$.
Moreover, $\alpha$ is a $*$-homomorphism,
since $\lim_{n\to \infty}\delta_n=0$ and 
$\bigcup_{n=0}^{\infty} Y_n$ is dense in $A_1$.

For any $n\in \mathbb{N}$ and $x\in X_n$,
\begin{align}\label{3.2.2}
\begin{split}
\| \alpha(x) - \alpha_n(x) \| 
\le \sum_{j=n}^{\infty} \| \alpha_{j+1}(x) - \alpha_j(x) \| 
\le \sum_{j=n}^{\infty} \frac{\gamma}{2^j} = \frac{\gamma}{2^{n-1}}.
\end{split}
\end{align}
Hence, by (e) in Lemma \ref{3.1},
\begin{align}\label{3.2.3}
\begin{split}
\| \alpha(x) \| 
&\ge \left| \| \alpha(x) - \alpha_n(x) \| - \| \alpha_n(x) \| \right| \\
&\ge \| \alpha_n(x)\| - \frac{\gamma}{2^{n-1}} \\
&= \| \theta_n(x) \| - \frac{\gamma}{2^{n-1}} \\
&\ge | \| \theta_n(x) - x \| - \| x \| | - \frac{\gamma}{2^{n-1}} \\
&\ge \| x \| - \eta - \frac{\gamma}{2^{n-1}}.
\end{split}
\end{align}
Let $n\to \infty$ in (\ref{3.2.3}).
Then, for any $x$ in the unit sphere of $A$, we have
\begin{align}
\| \alpha(x) \| \ge 1-\eta >0,
\end{align}
by the density of $\bigcup_{n=0}^{\infty}X_n$ in $A_1$.
Therefore, $\alpha$ is an injective map.

For any $b\in B_1$ and $n\in \mathbb{N}$, there exists $x\in A_1$ such that $\| x- u_n^* \cdots u_1^* b u_1 \cdots u_n \| <\gamma$.
\begin{align*}
\| \alpha_n(x) - b_j \| 
&= \| u_1\cdots u_n \theta_n(x) u_n^* \cdots u_1^* - b \| \\
&\le \| \theta_n(x) -x \| + \| x - u_n^* \cdots u_1^* b u_1 \cdots u_n \| \\
&< \eta + \gamma <1.
\end{align*}
Thus, $d(\alpha(A),B)<1$, that is, $\alpha$ is a surjective map by Proposition \ref{surjective}.

For any $x\in X$,
\begin{align*}
\| \alpha(x) - x \| 
\le \| \alpha(x) - \alpha_1(x) \| + \| \theta_1(x) - x \| \le \gamma +\eta <15 \gamma.
\end{align*}
by (\ref{3.2.3}) and  (e) in Lemma \ref{3.1}.
\end{proof}

\begin{thm}\label{3.3}
Let $C\subseteq D$ be a unital inclusion of  $\mathrm{C}^*$-algebras  and 
let $A$ and $B$ be separable intermediate $\mathrm{C}^*$-subalgebras for $C\subseteq D$ with a conditional expectation $E\colon D\to B$.
Suppose that $C\subseteq A$ is crossed product-like by a discrete amenable group
and $d(A,B)<\gamma<10^{-3}$.
Then
for any finite subset $X\subseteq A_1$ and finite set $Y\subseteq B_1$,
there exists a $C$-fixed surjective $*$-isomorphism $\alpha \colon A\to B$ such that 
\[
\alpha\approx_{X, 15\gamma} \id_A \ \  and \ \ \alpha^{-1} \approx_{Y, 17\gamma} \id_B.
\]
\end{thm}

\begin{proof}
There exists a finite set $\tilde{X}\subseteq A_1$ such that $\tilde{X}\subset_{\gamma} Y$.
By Proposition \ref{3.2},
there exists a $C$-fixed surjective $*$-isomorphism $\alpha\colon A\to B$
such that
\[
\alpha \approx_{X\cup \tilde{X}, 15\gamma} \id_A.
\]
Fix $y\in Y$.
Since $\tilde{X}\subset_{\gamma}Y$, there exists $\tilde{x}\in \tilde{X}$ such that $\| \tilde{x} - y \| <\gamma$.
Then, we have 
\begin{align*}
\| \alpha^{-1}(y) - y\| 
&\le \| \alpha^{-1}(y) - \tilde{x}\| + \| \tilde{x} - y\| \\
&\le \| y- \alpha(\tilde{x}) \| +  \gamma \\
&\le \| y - \tilde{x} \| + \| \tilde{x} - \alpha(\tilde{x}) \| +\gamma \\
&\le \gamma + 15\gamma + \gamma \le 17 \gamma.
\end{align*}
Therefore, $\alpha^{-1} \approx_{Y,17\gamma} \id_B$.
\end{proof}


\section{Crossed product von Neumann algebras by amenable groups}\label{von}


In Theorem \ref{2.3.5}, we now show that
given a unital inclusion $N\subseteq M$ of von Neumann algebras and 
intermediate von Neumann subalgebras $A,B$ for this inclusion 
with a normal conditional expectation form $M$ onto $B$,
if $A=N\rtimes G$, where $G$ is a discrete amenable group, 
and if $A$ and $B$ are sufficiently close,
then there exists a unitary $u \in N' \cap M$ such that $u A u^* = B$.
This unitary can be chosen to be close to the identity.

\begin{defn}\upshape
Let $N\subseteq M$ be a unital  inclusion of von Neumann algebras in $\mathbb{B}(H)$.
Then we call the inclusion $N\subseteq M$ is {\it crossed product-like}
if 
there exists a discrete group $U$ in $\mathcal{N}_M(N)$ such that
$M= (N \cup U)''$.
\end{defn}

\begin{exam}\upshape
Let $G$ be a discrete amenable group acting on a von Neumann algebra $N$.
Then an  inclusion $N\subseteq N\rtimes G$ is crossed product-like by $\{\lambda_g\}_{g\in G}$.
\end{exam}

\begin{exam}\upshape
Let $A\subseteq B$ be a crossed product-like inclusion of $\C^*$-algebras
acting non-degenerately on $H$.
Then an  inclusion $\bar{A}^{\w}\subseteq \bar{B}^{\w}$ of von Neumann algebras is crossed product-like.
\end{exam}

\begin{rem}\upshape
Let $N\subseteq M$ be a crossed product-like inclusion of von Neumann algebras in $\mathbb{B}(H)$
by a discrete amenable group $U\subseteq \mathcal{N}_M(N)$.
Then there is a left-invariant mean $m\colon \ell^{\infty}(U)\to \mathbb{R}$ with 
a net of finite subsets $\{ F_{\mu} \} \subseteq U$ such that 
\[
\lim_{\mu} \frac{1}{|F_{\mu}| } \sum_{g\in F_{\mu}} f(g) = m(f) , \ \ f\in \ell^{\infty}(U).
\]
Given a linear bounded map $\phi\colon U\to \mathbb{B}(H)$.
For $\xi,\eta\in H$, define $\phi_{\xi,\eta}\in \ell^{\infty}(U)$ by
\[
\phi_{\xi,\eta}(u)=\langle \phi(u) \xi, \eta \rangle, \ \ u\in U.
\]
Then there is an operator $T_{\phi}\in \mathbb{B}(H)$ which we will often write in the form 
\[
T_{\phi}= \int_{u\in U} \phi(u) d m
\]
 such that
\[
\langle T_{\phi}\xi,\eta \rangle = m( \phi_{\xi,\eta} )= \int_{u\in U} \langle \phi(u) \xi,\eta \rangle d m , \ \ \xi,\eta\in H.
\]
By the construction of $m$, we have 
\[
\langle T_{\phi} \xi, \eta \rangle = \lim_{\mu}  \frac{1}{|F_{\mu}|} \sum_{u\in F_{\mu}} \left\langle\phi(u) \xi,\eta \right\rangle, \ \ \xi,\eta\in H,
\]
that is, $T_{\phi}\in \overline{\mathrm{conv}}^{\mathrm{w}} \{ \phi(u) : u \in U \}$.
Furthermore,
\begin{align*}
\| T_{\phi} \|
&= \sup_{\xi,\eta\in H} \left| \int_{u\in U} \langle \phi(u) \xi,\eta \rangle d m \right| \\
&\le  \int_{u\in U} \sup_{\xi,\eta\in H} \left| \langle \phi(u) \xi,\eta \rangle \right| d m 
=\int_{u\in U} \| \phi(u) \| d m.
\end{align*}
\end{rem}

In the next lemma,
we shall find a unital normal $*$-homomorphism between von Neumann algebras.
This lemma is originated in Christensen's work \cite[Lemma 3.3]{Chris2},
 which discusses the perturbation theory for injective von Neumann algebras.

\begin{lemma}\label{2.2}
Let $N\subseteq M$ be an inclusion of von Neumann algebras in $\mathbb{B}(H)$
and 
let $A,B$ be intermediate von Neumann subalgebras for $N\subseteq M$
with a normal conditional expectation $E \colon M\to B$. 
Suppose that $N\subseteq A$ is crossed product-like by a discrete amenable group $U$ 
and $d(A,B)<\gamma<1/4$.
Then
there exists a unital $N$-fixed normal $*$-homomorphism $\Phi \colon A \to B$ 
such that
\[
\| \Phi -\id_A \| \le ( 8 \sqrt{2}+2)\gamma .
\]
\end{lemma}

\begin{proof}
Let $A_0:= \mathrm{span} \{ x u : x\in N, u\in U\}$.
By Stinespring's theorem,
there exist a Hilbert space $\tilde{H} \supseteq H$ and 
a unital normal $*$-homomorphism $\pi \colon  M \to \mathbb{B}(\tilde{H})$
such that
\begin{align*}
E (x)= P_H \pi (x) |_H, \ \ x\in M .
\end{align*}
Let $m\colon \ell^{\infty}(U)\to \mathbb{R}$ be a left-invariant mean with 
a net of finite subsets $\{ F_{\mu} \} \subseteq U$ such that 
\[
\lim_{\mu} \frac{1}{|F_{\mu}| } \sum_{g\in F_{\mu}} f(g) = m(f) , \ \ f\in \ell^{\infty}(U).
\]
Define
\[
t:= \int_{u\in U} \pi(u)P_H \pi(u^*) d m.
\]
Since $P_H\in \pi( N )'$, we have $t\in \pi(N )'$.
Fix $x\in A_0$.
Then there exist $\{ u_1,\dots, u_N\} \subseteq U$ and $\{ x_1,\dots,x_N\}\subseteq N_1$ 
such that $x=\sum_{i=1}^N x_i u_i $.
For any $\xi,\eta\in H$,
\begin{align*}
\left\langle \pi(x) t \xi, \eta \right\rangle
&=\int_{u\in U} \left\langle \pi(x) \pi(u) P_H \pi(u^*) \xi, \eta \right\rangle d m \\
&=\int_{u\in U} \sum_{i=1}^N \left\langle \pi( x_i u_i  u)P_H\pi(u^*) \xi, \eta \right\rangle d m \\
&=\sum_{i=1}^N \int_{v\in U} \left\langle \pi(x_i v)P_H\pi(v^* u_i  ) \xi, \eta \right\rangle d (u_i^* m) \\
&=\sum_{i=1}^N \int_{v\in U} \left\langle \pi( x_i v)P_H\pi(v^* u_i  ) \xi, \eta \right\rangle d m \\
&=\sum_{i=1}^N \int_{v\in U} \left\langle \pi(v)P_H\pi(v^* x_i u_i  ) \xi, \eta \right\rangle d m \\
&=\int_{u\in U} \left\langle \pi(u) P_H \pi(u^*) \pi(x) \xi, \eta \right\rangle d m 
= \left\langle t \pi(x) \xi, \eta \right\rangle.
\end{align*}
Therefore, $t\in \pi( A)'$ by the normality of $\pi$.
Furthermore,
for $u\in U$, there is $v\in B_1$ such that $\| u- v\| <\gamma$.
Then since $P_H\in \pi(B)'$, we have
\[
\| \pi (u)P_H - P_H\pi(u)\| \le \| \pi(u) P_H- \pi(v)P_H\| + \| P_H\pi(v) - P_H\pi(u)\| <2\gamma.
\]
Therefore,
\begin{align*}
\| t - P_H\| 
&\le \int_{u\in U} \| \pi(u)P_H \pi(u^*) - P_H\| d m \\
&=\int_{u\in U} \| \pi(u) P_H -  P_H \pi(u) \| d m 
\le 2 \gamma <\frac{1}{2}.
\end{align*}
Define $\delta:=\| t  - P_H\|$ and $q:=\chi_{[1-\delta,1]}(t)$.
Since $\|q-P_H\|\le 2\delta<1$,
there exists a unitary $w\in C^*(t, P_H, I_{\tilde{H}})$ such that $w P_H w^*=q$ and 
$\| w- I_{\tilde{H}} \| \le 2 \sqrt{2}\delta$ by Lemma \ref{polar}\,(3).
Define a map $\Phi\colon A \to \mathbb{B}(\tilde{H})$ by
\[
\Phi(x):= P_H w^* \pi (x) w |_H, \ \ x\in A .
\]
Since $t\in \overline{\mathrm{conv}}^{\w} \{ \pi(u) P_H \pi(u^*) : u\in U\}$ and 
$P_H \pi( A ) |_H= E( A )\subseteq B$,
we have $\Phi( A )\subseteq B$.
For any $x,y\in A$,
\begin{align*}
\Phi(x)\Phi(y)
&=P_H w^*\pi(x) w P_H w^*\pi(y)w P_H \\
&=P_H w^* \pi(x) q \pi(y) w P_H \\
&=P_H w^* q \pi(x y)w P_H \\
&=P_H w^* \pi (x y) w P_H =\Phi(x y).
\end{align*}
Therefore, $\Phi$ is a $*$-homomorphism.

Furthermore, for any $x\in A_1$,
\begin{align*}
\| \Phi(x) - x \| 
&\le \| \Phi(x) - E(x)\| + \| E (x) - x\| \\
&\le \| P_H w^* \pi(x) w P_H - P_H \pi(x) P_H \| + 2d(A , B) \\
&\le 2\| w - I_{\tilde{H}}\| + 2 d(A , B) \\
&\le (8 \sqrt{2}+2) \gamma.
\end{align*}
Since $w\in C^*(t, P_H, I_{\tilde{H}})\subseteq \pi(N)'$, 
$\Phi$ is a $N$-fixed map.
\end{proof}

We base the next lemma on Christensen's work \cite[Propositions 4.2 and 4.4]{Chris2}, 
which show similar results for injective von Neumann algebras.

\begin{lemma}\label{2.3}
Let $A,B$ and $N$ be von Neumann algebras  in $\mathbb{B}(H)$ with $N\subseteq A\cap B$.
Suppose that 
$N\subseteq A$ is crossed product-like by a discrete amenable group $U$.
Then
given two unital $N$-fixed normal $*$-homomorphisms 
$\Phi_1,\Phi_2 \colon A \to B$ with 
$\| \Phi_1-\Phi_2\| <1$,
there exists a unitary $u\in N' \cap B$ such that 
$\Phi_1=\mathrm{Ad}(u)\circ \Phi_2$ and 
$\| u - I \| \le \sqrt{2}\| \Phi_1- \Phi_2\|$.
\end{lemma}

\begin{proof}
Let  $A_0:= \mathrm{span}\{ x u : x\in N ,u\in U\}$
and let $m\colon\ell^{\infty}(U)\to \mathbb{R}$ be a left-invariant mean with 
there is a net of finite subsets $\{ F_{\mu} \} \subseteq U$ such that $m_{\mu}\to m$ in the weak-$*$ topology, where
\[
m_{\mu}(f)= \frac{1}{|F_{\mu}| } \sum_{g\in F_{\mu}} f(g), \ \ f\in \ell^{\infty}(U).
\]
Define 
\[
s := \int_{u\in U} \Phi_1(u) \Phi_2(u^*) d m .
\]
Since $\Phi_1$ and $\Phi_2$ are $N$-fixed maps and $U\subseteq \mathcal{N}_A(N)$,
we have $s \in N ' \cap B$. 
For $x\in A_0$,
there exist $\{ u_1,\dots,u_N\} \subseteq U$ and $\{ x_1,\dots ,x_N\} \subseteq N$ such that $x=\sum_{i=1}^N x_i u_i $.
For any $\xi ,\eta\in H$,
\begin{align*}
\langle \Phi_1(x)s\xi,\eta \rangle
&=\int_{u\in U} \left\langle \Phi_1(x) \Phi_1(u)\Phi_2(u^*) \xi, \eta \right\rangle d m \\
&=\int_{u\in U} \sum_{i=1}^N \left\langle \Phi_1(c_i u_i u) \Phi_2(u^* ) \xi, \eta\right\rangle d m \\
&=\sum_{i=1}^N \int_{v\in U} \left\langle \Phi_1(c_i v) \Phi_2(v^* u_i ) \xi, \eta\right\rangle d (u_i^* m) \\
&=\sum_{i=1}^N \int_{v\in U} \left\langle \Phi_1(c_i v) \Phi_2(v^* u_i ) \xi, \eta\right\rangle d  m \\
&=\sum_{i=1}^N \int_{v\in U} \left\langle \Phi_1( v) \Phi_2(v^* c_i u_i ) \xi, \eta\right\rangle d  m \\
&= \int_{v\in U} \left\langle \Phi_1(v) \Phi_2(v^* x) \xi, \eta\right\rangle d m 
=\langle s \Phi_2(x) \xi,\eta \rangle.
\end{align*}
Therefore, by the normality of $\Phi_1$ and $\Phi_2$,
\begin{align}\label{2.3.1}
\Phi_1(x)s= s\Phi_2(x), \ \ x\in A.
\end{align}
By taking adjoint,
\begin{align}\label{2.3.2}
s^*\Phi_1(x)=\Phi_2(x)s^*, \ \ x\in A.
\end{align}
By (\ref{2.3.1}) and (\ref{2.3.2}),
for $x\in A$,
$s^* s \Phi_2(x)= s^* \Phi_1(x) s= \Phi_2(x) s^* s$.
Thus, 
\begin{align}\label{2.3.3}
|s|^{-1} \Phi_2(x)= \Phi_2(x)|s|^{-1}, \ \ x\in A.
\end{align}

Furthermore,
\begin{align*}
\|s- I_H\| 
\le \int_{u\in U} \| \Phi_1(u)\Phi_2(u^*)- \Phi_1(u)\Phi_1(u^*) \| d m 
\le \| \Phi_1-\Phi_2\| <1.
\end{align*}
Hence by Lemma \ref{polar}\,(1), we can choose the unitary 
$u\in C^*(s,I )\subseteq N'\cap B$ 
in the polar decomposition of $s$ with $\| u- I  \| \le \sqrt{2}\| s- I \|$.

By (\ref{2.3.1}) and (\ref{2.3.3}),
\begin{align*}
\Phi_1(x)u=\Phi_1(x)s |s|^{-1}= s\Phi_2(x) |s|^{-1}= s |s|^{-1} \Phi_2(x) =u \Phi_2(x), \ \ x\in A.
\end{align*}
Therefore, $\Phi_1=\mathrm{Ad}(u)\circ \Phi_2$.
\end{proof}

Using Lemma \ref{2.2} and \ref{2.3},
it follows Theorem C.

\begin{thm}\label{2.3.5}
Let $N\subseteq M$ be an inclusion of von Neumann algebras in $\mathbb{B}(H)$ and 
let $A,B$ be intermediate von Neumann subalgebras for $N \subseteq M$
with a normal conditional expectation from $M$ onto $B$. 
Suppose that $N\subseteq A$ is crossed product-like by a discrete amenable group
and $d(A,B)<\gamma<10^{-2}$.
Then there exists a unitary $u\in N ' \cap (A\cup B)''$ 
such that $u A u^*= B$ 
and $\| u - I \| \le 2(8+\sqrt{2})\gamma$.
\end{thm}

\begin{proof}
By Lemma \ref{2.2}, 
there exists a unital $N$-fixed normal $*$-\hspace{0pt}homomorphism $\Phi\colon  A \to B$ such that
\[
\| \Phi- \id_A \| \le (8 \sqrt{2}+2)\gamma.
\]
Since $(8 \sqrt{2}+2)\gamma <1$,
there exists a unitary $u\in N' \cap (A\cup B)''$ such that 
$\Phi=\mathrm{Ad}(u)$ and 
$\| u - I \| \le \sqrt{2}\| \Phi- \id_A \| $ by Lemma \ref{2.3}.
Thus,
\[
u A u^* = \Phi(A) \subseteq B.
\]

Fix $x\in B_1$.
There exists $y \in A_1$ such that $\| x -y \|\le \gamma$.
Then,
\begin{align*}
\| y - u x u^*\| \le \| y- x\| + \| x- u x u^*\| \le \gamma + 2\| u- I\| <1.
\end{align*}
Thus, $d(u A u^*, B )<1$, that is, $u A u^*= B $ by Proposition \ref{surjective}.
\end{proof}

\begin{cor}\label{2.4}
Let $C\subseteq D$ be a unital inclusion of $\mathrm{C}^*$-algebras and 
let $A,B$ be intermediate $\mathrm{C}^*$-subalgebras for $C\subseteq D$
with a conditional expectation $E\colon D\to B$. 
Suppose that $C\subseteq A$ is crossed product-like by a discrete amenable group
and $d(A,B)<\gamma<10^{-2}$.
Then there exists a unitary $u\in (C^{**})' \cap W^*(A^{**}, B^{**})$ 
such that $u A^{**} u^*= B^{**}$ 
and $\| u - I \| \le 2(8+\sqrt{2})\gamma$.
\end{cor}

\begin{proof}
By  a general construction, there exists a normal conditional expectation $E ^{**}\colon D^{**}\to B^{**}$.
Let $(\pi, H)$ be the universal representation of $D$ and 
identify $A^{**}$, $B^{**}$, $C^{**}$ and $D^{**}$ 
with $\pi(A)''$, $\pi(B)''$, $\pi(C)''$ and $\pi(D)''$, respectively.
Then by Theorem \ref{2.3.5} and Lemma \ref{weak-closure}, the corollary follows.
\end{proof}


\section{Unitary equivalence}\label{unitary}


In this section, 
we  show the fourth main result: Theorem D.
For a unital inclusion $C\subseteq D$ of C$^*$-algebras  acting on a separable  Hilbert space $H$
and sufficiently close separable intermediate C$^*$-subalgebras $A$, $B$ for $C\subseteq D$ with a conditional expectation of $D$ onto $B$,
if  $A=C \rtimes G$ with  $G$ discrete amenable and 
if $C'\cap A$ is weakly dense in $C'\cap \overline{A}^{\w}$, then $A$ and $B$ are unitarily equivalent.
The unitary can be chosen in the relative commutant of $C'\cap (A\cup B)''$.
To show this,
we modify the arguments of Section 5 in Christensen et al. \cite{CSSWW}.

\begin{lemma}\label{4.1}
Let $C\subseteq D$ be a unital inclusion of $\mathrm{C}^*$-algebras 
acting non-degenerately on a separable Hilbert space $H$.
Let $A$ and $B$ be separable intermediate $\mathrm{C}^*$-subalgebras for $C\subseteq D$ 
with a conditional expectation $E\colon D\to B$.
Suppose that
$C\subseteq A$ is crossed product-like by a discrete amenable group
and $C'\cap C^*(A,B)\subseteq \overline{C'\cap A}^{\w}$.
If $d(A,B)<\gamma<10^{-4}$,
then for any finite subsets $X\subseteq B_1$, $Z_A\subseteq A_1$ and $\varepsilon,\mu>0$,
there exist finite subsets $Y\subseteq B_1$, $Z\subseteq A_1$, a positive constant $\delta>0$, a unitary $u\in C'\cap C^*(A,B)$ 
and a $C$-fixed surjective $*$-isomorphism $\theta\colon B\to A$ with the following conditions$\colon$
\begin{enumerate}
\item[\upshape{(i)}] 
$\delta<\varepsilon;$

\item[\upshape{(ii)}] 
$X\subseteq_{\varepsilon}Y;$

\item[\upshape{(iii)}] 
$\| u - I\| \le 75\gamma;$

\item[\upshape{(iv)}] 
$\theta\approx_{Y,\delta} \mathrm{Ad}(u);$

\item[\upshape{(v)}] 
$\theta \approx_{X,117\gamma} \id_B$ and 
$\theta^{-1} \approx_{Z_A, 115\gamma} \id_A;$

\item[\upshape{(vi)}] 
For any $C$-fixed surjective $*$-isomorphism $\phi\colon B\to A$ with 
$
\phi^{-1}\approx_{Z, 365\gamma} \id_A,
$
there exists a unitary $w\in C'\cap A$ such that 
\[
\mathrm{Ad}(w) \circ \phi \approx_{Y, \delta/2} \theta \ \ \text{and} \ \ \| w - u\| \le 665\gamma;
\]

\item[\upshape{(vii)}] 
For any finite subset $S\subseteq H_1$ and 
unitary $v\in C'\cap C^*(A,B)$ with $\mathrm{Ad} (v) \approx_{Y,\delta}\theta$ and $\| v -u\| \le 740\gamma$,
there exists a unitary $\tilde{v} \in C'\cap A$ such that $\mathrm{Ad}(\tilde{v} v) \approx_{X,\varepsilon} \theta$, 
$\| \tilde{v}- I\| \le 740\gamma$ and
\[
\| (\tilde{v}v-u)\xi \| <\mu \ \ and \ \ \| (\tilde{v}v-u)^* \xi\| <\mu, \ \ \xi \in S.
\]
\end{enumerate}
\end{lemma}

\begin{proof}
Let  $X, Z_A, \varepsilon$ and $\mu$ be given.
By Lemma \ref{1.5},
there exists a finite subset $Z_1\subseteq B_1$ satisfying the following condition:
given two unital $C$-fixed $*$-homomorphisms $\phi_1,\phi_2\colon B\to B$ with $\phi_1\approx_{Z_1,32\gamma}\phi_2$,
there exists a unitary $w_1\in C'\cap B$ such that 
$\phi_1\approx_{X,\varepsilon/3}\mathrm{Ad} (w_1) \circ \phi_2$ and 
$\| w_1- I_H\| \le 32\sqrt{2}\gamma$.

By Proposition \ref{3.2},
there exists a $C$-fixed surjective $*$-isomorphism $\beta\colon B\to A$ such that 
\begin{align}\label{4.1.1}
\beta\approx_{Z_1, 17\gamma} \id_B.
\end{align}
Let $X_0:=\beta(X)$.

By Lemma \ref{1.8},
there exist a finite set $Y_0\subseteq A_1$ and $\delta>0$ with the following properties:
$\delta<\varepsilon/6$,
$X_0\subseteq Y_0$
and 
given a finite set $S_0\subseteq H_1$ and a unitary $u\in C'\cap C^*(A,B)$ with
$\| u - I\| \le 740\gamma$ and 
\begin{align*}
\| u y_0 - y_0 u \| \le 3\delta, \ \ y_0\in Y_0,
\end{align*}
there exists a unitary $v\in C'\cap A$ such that $\| v -I\| \le 740\gamma$,
\[
\| v x_0 - x_0 v \| \le \frac{\varepsilon}{6}, \ \ x_0\in X_0
\]
and
\[
\| (v-u)\xi_0\| <\mu \ \ \mathrm{and} \ \ \|(v-u)^*\xi_0\|<\mu, \ \ \xi_0\in S_0,
\]
since $C'\cap C^*(A,B)\subseteq \overline{C'\cap A}^{\w}$.

By Lemma \ref{1.5},
there exists a finite set $Z\subseteq A_1$ with the following properties:
$\beta(Z_1)\cup Z_A\subseteq Z$ and given $\gamma_0<1/10$ and 
two unital $C$-fixed $*$-homomorphism $\phi_1,\phi_2\colon A\to C^*(A,B)$ 
with $\phi_1\approx_{Z,\gamma_0}\phi_2$,
there exists a unitary $u_0\in C'\cap C^*(A,B)$ such that $\mathrm{Ad} (u_0) \circ \phi_1\approx_{Y_0,\delta/2}\phi_2$
and $\| u_0 - I\| \le \sqrt{2}\gamma_0$.

By Proposition \ref{3.3},
there exists a $C$-fixed surjective $*$-isomorphism $\sigma\colon A\to B$ such that
\begin{align}\label{4.1.2}
\sigma\approx_{Z,15\gamma} \id_A \ \ \mathrm{and} \ \ \sigma^{-1}\approx_{X,17\gamma} \id_B.
\end{align}
Hence, by the choice of $Z$,
there exists a unitary $u_0\in C'\cap C^*(A,B)$ such that 
\begin{align}\label{4.1.3}
\sigma\approx_{Y_0,\delta/2} \mathrm{Ad} (u_0)
\end{align}
and $\| u_0 - I\| \le 15\sqrt{2}\gamma <25\gamma$.

Since $\beta(Z_1)\subseteq Z$, (\ref{4.1.1}) and (\ref{4.1.2}),
we have
\begin{align}\label{4.1.4}
\sigma\circ\beta\approx_{Z_1,32\gamma} \id_B.
\end{align}
By the definition of $Z_1$,
there exists a unitary $w_1\in C'\cap B$ such that 
\begin{align}\label{4.1.5}
\sigma\circ \beta \approx_{X,\varepsilon/3} \Ad  (w_1)
\end{align}
and $\| w_1 - I\| \le 32\sqrt{2}\gamma < 50\gamma$.

Now define $\theta:=\sigma^{-1}\circ \Ad (w_1)$, $Y:=\theta^{-1}(Y_0)$ and $u:= u_0^*w_1$.
Fix $y\in Y$.
Let $y_0:=\theta(y)\in Y_0$.
Then,
\begin{align*}
\| \theta(y) - \Ad (u) (y)\| 
&=\| y_0 - (\Ad  (u) \circ\theta^{-1})(y_0) \| \\
&=\| y_0- (\Ad (u_0^*) \circ \sigma)(y_0) \| \\
&=\| \Ad (u_0) (y_0) - \sigma(y_0) \| \le \frac{\delta}{2},
\end{align*}
since $\theta^{-1}=\Ad (w_1^*) \circ \sigma$ and (\ref{4.1.3}).
Thus, $\theta\approx_{Y,\delta/2}\Ad (u)$, so that condition (iv) holds.

By the definition of $u$, we have
\begin{align*}
\| u- I\| = \| w_1- u_0\| \le \| w_1- I\| + \| I - u_0 \| <75 \gamma.
\end{align*}
Hence, condition (iii) holds.

For any $x\in X$,
\begin{align}\label{4.1.6}
\begin{split}
\| \theta(x) - x \| 
&\le \| (\sigma^{-1}\circ\Ad (w_1))(x) - \sigma^{-1}(x) \| + \| \sigma^{-1}(x) - x\| \\
&\le 2\| w_1- I_H\| +17\gamma \\
&\le 100\gamma +17\gamma= 117\gamma.
\end{split}
\end{align}
For any $z\in Z$,
\begin{align*}
\| \theta^{-1}(z) - z \| 
&\le \| (\Ad (w_1^*) \circ \sigma)(z) - \Ad (w_1^*)(z) \| + \| \Ad (w_1^*)(z) - z\| \\
&\le \| \sigma(z) - z\| + 2\| w_1-I\| \\
&\le 15\gamma +100\gamma=115\gamma.
\end{align*}
Therefore,
\begin{align}\label{4.1.7}
\theta^{-1}\approx_{Z,115\gamma} \id_A.
\end{align}
Since $Z_A\subseteq Z$,
we have $\theta^{-1}\approx_{Z_A,115\gamma} \id_A$, so that condition (v) holds.

By (\ref{4.1.5}),
\begin{align}\label{4.1.8}
\theta=\sigma^{-1}\circ \Ad(w_1) \approx_{X,\varepsilon/3} \sigma^{-1}\circ \sigma\circ\beta=\beta
\end{align}
Fix $x_0\in X_0$.
Let $x:= \beta^{-1}(x_0)\in X$.
Then, by (\ref{4.1.8}),
\begin{align*}
\| \theta^{-1}(x_0)-\beta^{-1}(x_0)\| = \| (\theta^{-1}\circ\beta)(x) - x \| =\| \beta(x) -\theta(x)\| \le \frac{\varepsilon}{3}.
\end{align*}
Therefore, 
\begin{align}\label{4.1.9}
\theta^{-1}\approx_{X_0, \varepsilon/3} \beta^{-1}.
\end{align}
Hence,
\[
X=\beta^{-1}(X_0)\subseteq_{\varepsilon/3} \theta^{-1}(X_0) \subseteq \theta^{-1}(Y_0) =Y,
\]
so that condition (ii) holds.

We now verify condition (vi).
Let $\phi\colon B\to A$ be a $C$-fixed surjective $*$-isomorphism with $\phi^{-1}\approx_{Z,365\gamma} \id_A$.
By (\ref{4.1.2}),
\[
\phi^{-1}\approx_{Z,380\gamma}\sigma.
\]
Thus, by the definition of $Z$,
there exists a unitary $w_0\in C'\cap B$ such that 
\begin{align}\label{4.1.10}
\Ad (w_0)\circ \phi^{-1}\approx_{Y_0,\delta/2}\sigma
\end{align}
and $\| w_0 - I\| \le 380\sqrt{2}\gamma<540\gamma$.
Fix $y\in Y$.
Let $y_0:= \theta(y)\in Y_0$.
Then, since $w_0^*w_1\in B$, we have
\begin{align*}
\| \theta(y) - (\Ad(\phi(w_0^*w_1))\circ \phi)(y) \| 
&= \| \theta(y)- (\phi\circ\Ad(w_0^*w_1))(y) \| \\
&= \| y_0 - (\phi\circ\Ad (w_0^*) \circ\sigma)(y_0)\| \\
&=\| (\Ad (w_0) \circ\phi^{-1})(y_0)-\sigma(y_0)\| \le \frac{\delta}{2}
\end{align*}
by (\ref{4.1.10}).
Define $w:=\phi(w_0^*w_1)$ so
$\theta\approx_{Y,\delta/2}\Ad (w) \circ\phi$.
Since $\phi$ is $C$-fixed map and $w_0,w_1\in C'$, $w$ is in $C'\cap A$.
Moreover,
\begin{align*}
\| w-u\| 
&\le \| w -I\|+ \| I-u\|
\le \| w_0^*w_1-I\| +  75\gamma \\
&\le \|w_0-I\|+\|I-w_1\|+75\gamma
\le (540+50+75)\gamma \\
&=665\gamma.
\end{align*}
Therefore, condition (vi) is proved.

It only remains to prove condition (vii).
Let $S\subseteq H_1$ be a finite set and $v\in C'\cap C^*(A,B)$ be a unitary with $\| v-u\| \le 740\gamma$ and 
\begin{align}\label{4.1.11}
\Ad (v)\approx_{Y,\delta}\theta.
\end{align}
Fix $y_0\in Y_0$.
Let $y:=\theta^{-1}(y_0)\in Y$.
Then, 
\begin{align}\label{4.1.99}
\begin{split}
\| \sigma(y_0)-\Ad(w_1v^*)(y_0)\|
&=\| (\Ad (w_1^*) \circ \sigma)(y_0)- \Ad (v^*) (y_0) \| \\
&=\| y - (\Ad (v^*) \circ\theta)(y) \| \\
&=\| \Ad (v) (y) - \theta(y) \| \le\delta .
\end{split}
\end{align}
This and (\ref{4.1.3}) give $\Ad (u_0) \approx_{Y_0,3\delta/2}\Ad(w_1v^*)$.
Therefore, for any $y_0\in Y_0$,
\begin{align}\label{4.1.12}
\| (v w_1^* u_0) y_0 - y_0(v w_1^* u_0) \| =\|u_0 y_0u_0^*- (w_1 v^*) y_0 (w_1 v^*)^* \| \le \frac{3}{2}\delta.
\end{align}
Furthermore,
\begin{align}\label{4.1.13}
\| v w_1^* u_0 - I\| =\| w_1^* u_0- v^*\| = \| u^* - v^*\| \le 740\gamma.
\end{align}
Let $S_0:=S\cup w_1 S\cup v S$.
By the definition of $Y_0$ and $\delta$, with $v w_1^* u_0$ and $S_0$,
there exists a unitary $v_0\in C'\cap A$ such that $\| v_0-I\|\le 740\gamma$, 
\begin{align}\label{4.1.14}
\| v_0 x_0- x_0 v_0 \| \le \frac{\varepsilon}{6}, \ \ x_0\in X_0
\end{align}
and 
\begin{align}\label{4.1.15}
\| (v_0- v w_1^*u_0)\xi_0\| <\mu \ \ \mathrm{and} \ \ \| (v_0- v w_1^*u_0)^*\xi_0\| <\mu, \ \ \xi_0\in S_0.
\end{align}
Let $\tilde{v}:=v_0^*$.
Then, $\| \tilde{v}-I\| \le \| v_0-I\| \le 740\gamma $.

For any $\xi\in S$,
\begin{align*}
\| (\tilde{v}v -u)\xi\| &=\| ( v_0^* v - u_0^* w_1)\xi \| =\| ( v_0^* - u_0^* w_1 v^*) v \xi\| \\
&=\| (v_0- v w_1 u_0)^* v \xi \| <\mu
\end{align*}
by (\ref{4.1.15}) and $v \xi \in S_0$.
Moreover,
\[
\| (\tilde{v} v -u)^*\xi\| =\| (v_0^*v- u_0^*w_1)^*\xi\| = \| v^*(v_0 - v w_1^* u_0)\xi\| <\mu
\]
by (\ref{4.1.15}).

For any $x_0\in X_0$, by (\ref{4.1.99}) and (\ref{4.1.14}),
\begin{align}\label{4.1.16}
\begin{split}
\| \theta^{-1}(x_0) -  \Ad(v^* v_0)(x_0)\| 
&\le \| \theta^{-1}(x_0) - \Ad(v^*)(x_0) \| + \| \Ad(v^*)(x_0) - \Ad(v^* v_0)(x_0) \| \\
&= \| (\Ad(w_1^*) \circ \sigma)(x_0) - \Ad(v^*)(x_0) \| + \| x_0 - \Ad(v_0)(x_0) \|   \\
&= \| \sigma(x_0) - \Ad(w_1v^*)(x_0) \| + \| v_0 x_0 - x_0 v_0 \| \\
&< \delta +  \frac{\varepsilon}{6} \le \frac{\varepsilon}{3}.
\end{split}
\end{align}
Let $x\in X$ and
$x_0:=\beta(x)\in X_0$.
By (\ref{4.1.8}), (\ref{4.1.9}) and (\ref{4.1.16}),
\begin{align*}
\| \Ad(\tilde{v}v)(x) - \theta(x) \|  
&\le \| \Ad(\tilde{v}v)(x) - \beta(x) \| + \| \beta(x)- \theta(x)\| \\
&\le \| (\Ad(\tilde{v}v)\circ \beta^{-1})(x_0) - x_0 \| + \frac{\varepsilon}{3} \\
&= \| \beta^{-1}(x_0) - \Ad(v^* v_0)(x_0) \| + \frac{\varepsilon}{3} \\
&\le \| \beta^{-1}(x_0) - \theta^{-1}(x_0) \| + \| \theta^{-1}(x_0) - \Ad(v^* v_0)(x_0)\| + \frac{\varepsilon}{3} \\
&\le \frac{\varepsilon}{3} + \frac{\varepsilon}{3} + \frac{\varepsilon}{3} = \varepsilon.
\end{align*}
Therefore, condition (vii) holds.
\end{proof}

\begin{lemma}\label{4.2}
Let $C\subseteq D$ be a unital inclusion of $\mathrm{C}^*$-algebras acting non-degenerately on a separable Hilbert space $H$.
Let $A$ and $B$ be separable intermediate $\mathrm{C}^*$-subalgebras for $C\subseteq D$ 
with a conditional expectation $E\colon D\to B$.
Let $\{ a_n\}_{n=1}^{\infty}$, $\{ b_n\}_{n=1}^{\infty}$ and $\{ \xi_n\}_{n=0}^{\infty}$ be dense subsets in $A_1$, $B_1$ and $H_1$, respectively.
Suppose that 
$C\subseteq A$ is crossed product-like by a discrete amenable group
and $C'\cap C^*(A,B)\subseteq \overline{C'\cap A}^{\w}$.
If $d(A,B)<\gamma<10^{-5}$,
then
there exist finite subsets 
$\{ X_n\}_{n=0}^{\infty}, \{Y_n\}_{n=0}^{\infty}\subseteq B_1$, 
$\{Z_n\}_{n=0}^{\infty}\subseteq A_1$, 
positive constants $\{\delta_n\}_{n=0}^{\infty}$, 
unitaries $\{u_n\}_{n=0}^{\infty}\subseteq C'\cap C^*(A,B)$ 
and $C$-fixed surjective $*$-isomorphisms $\{\theta_n\colon B\to A\}_{n=0}^{\infty}$ with the following conditions$\colon$
\begin{enumerate}
\item[\upshape{(1)}]
For $n\ge 1$, $b_1,\dots,b_n\in X_n;$

\item[\upshape{(2)}]
For $n\ge 0$, $X_n\subseteq_{2^{-n}/3}Y_n$ and $\delta_n<2^{-n};$

\item[\upshape{(3)}] 
For $n\ge 1$, $\theta_n\approx_{X_{n-1},2^{-(n-1)}}\theta_{n-1};$

\item[\upshape{(4)}]
For $n\ge 0$, $\theta_n\approx_{Y_n,\delta_n}\Ad (u_n);$

\item[\upshape{(5)}]
For $1\le j \le n$, $\|(u_n- u_{n-1})\xi_j\|<2^{-n}$ and $\| (u_n- u_{n-1})^* \xi_j \| <2^{-n};$

\item[\upshape{(6)}]
For $1\le j\le n$, there exists $x\in X_n$ such that $\| \theta_n(x)-b_j\| \le 9/10;$

\item[\upshape{(7)}] 
For $n\ge 0$ and
a $C$-fixed surjective $*$-isomorphism $\phi\colon B\to A$ with 
$\phi^{-1}\approx_{Z_n, 365\gamma} \id_A$,
there exists a unitary $w\in C'\cap A$ such that 
$\mathrm{Ad} (w) \circ \phi \approx_{Y_n, \delta_n/2} \theta_n$ and 
$\| w - u_n\| \le 665\gamma;$

\item[\upshape{(8)}] 
For $n\ge 0$,
a finite subset $S\subseteq H_1$ and 
a unitary $v\in C'\cap C^*(A,B)$ with $\mathrm{Ad} (v) \approx_{Y_n,\delta_n}\theta_n$ and
 $\| v -u_n\| \le 740\gamma$,
there exists a unitary $\tilde{v} \in C'\cap A$ such that 
$\mathrm{Ad}(\tilde{v} v) \approx_{X_n, 2^{-(n+1)}} \theta_n$, 
$\| \tilde{v}- I\| \le 740\gamma$ and
\[
\| (\tilde{v}v-u_n)\xi \| < \frac{1}{2^{n+1}} \ \ and \ \ \| (\tilde{v}v-u_n)^* \xi\| <\frac{1}{2^{n+1}}, \ \ \xi \in S;
\]

\item[\upshape{(9)}]
For $n\ge 0$, there is a unitary $z\in A$ such that $\| z - u_n\| \le 75\gamma$.
\end{enumerate}
\end{lemma}

\begin{proof}
We prove this lemma by using the induction.
Denote by (a)$_n$ the condition (a) for $n$.
Let $X_0=Y_0=Z_0=\emptyset$, $\delta_0=1/2$, $u_0=I$ and 
$\theta_0\colon B\to A$ be any $C$-fixed surjective $*$-isomorphism by Proposition \ref{3.2}.
Then,
conditions (1)$_0$, (3)$_0$, (5)$_0$ and (6)$_0$ do not be defined.
Conditions (2)$_0$ and (4)$_0$ are clear, since $X_0=Y_0=\emptyset$.
In conditions (7)$_0$, (8)$_0$ and  (9)$_0$,
by taking $w=I$, $\tilde{v}=v^*$ and $z=I$,
that conditions are satisfied.

Assume the statement holds for $n$;
we will prove it for $n+1$.
By (9)$_n$,
there exists a unitary $z\in A$ such that $\| z - u_n\| \le 75\gamma$.
For $1\le j\le n+1$,
there exists $x_j\in B_1$ such that $\| x_j- z^* a_j z\| \le \gamma$.
Define $X_{n+1}:=X_n\cup Y_n\cup \{b_n\}\cup \{x_1,\ldots,x_{n+1}\}$.
In Lemma \ref{4.1},
let $X=X_{n+1}$, $Z_A=Z_n$, $\varepsilon=\delta_n/6$ and $\mu=2^{-(n+2)}$
and so there exist $Y_{n+1}\subseteq B_1$, $Z_{n+1}\subseteq A_1$, $\delta_{n+1}>0$, 
$u\in C'\cap C^*(A,B)$ and $\theta\colon B\to A$ 
with conditions (i)-(vii) of that lemma.
By Lemma \ref{4.1} (i), $\delta_{n+1}<\varepsilon=\delta_n/6<2^{-(n+1)}/3$.
By Lemma \ref{4.1} (ii), $X_{n+1}\subseteq_{2^{-(n+1)}/3} Y_{n+1}$.
Thus, condition (2)$_{n+1}$ holds.

By applying $\theta$ to condition (7)$_n$,
we may find a unitary $w\in C'\cap A$ such that 
\begin{align}\label{4.2.1}
\Ad (w) \circ \theta\approx_{Y_n,\delta_n/2}\theta_n
\end{align}
and 
$\| w-u_n\|\le665\gamma$.

Fix $y\in Y_n$.
Since $Y_n\subseteq X_{n+1}\subseteq_{\delta_n/6}Y_{n+1}$,
there exists $\tilde{y}\in Y_{n+1}$ such that $\| y-\tilde{y}\|\le \delta_n/6$.
Then, by Lemma \ref{4.1} (iv),
\begin{align*}
\| \Ad (u)(y)- \theta(y)\| 
&\le \| \Ad (u) (y)-\Ad (u) (\tilde{y})\|+\|\Ad (u) (\tilde{y})-\theta(\tilde{y})\|+\|\theta(\tilde{y})-\theta(y)\| \\
&\le \frac{\delta_n}{6}+\delta_{n+1}+\frac{\delta_n}{6}\le \frac{\delta_n}{2}. 
\end{align*}
This and (\ref{4.2.1}) give
\begin{align}\label{4.2.2}
\Ad(w u)\approx_{Y_n,\delta_n} \theta_n.
\end{align}
Moreover,
\begin{align}\label{4.2.3}
\begin{split}
\| w u - u_n\|
\le \|  w (u - I)  \| + \| w - u_n \| 
\le 75\gamma+665\gamma=740\gamma.
\end{split}
\end{align}
By (\ref{4.2.2}) and (\ref{4.2.3}), 
we can apply $w u$ and $\{\xi_1,\ldots,\xi_{n+1}\}$ to condition (8)$_n$.
Hence, 
there exists a unitary $\tilde{v}\in C'\cap A$ such that 
\begin{align}\label{4.2.4}
\mathrm{Ad}(\tilde{v} w u) \approx_{X_n, 2^{-(n+1)}} \theta_n,
\end{align} 
$\| \tilde{v}- I\| \le 740\gamma$ and
\begin{align}\label{4.2.5}
\| (\tilde{v}w u-u_n)\xi_j \| < \frac{1}{2^{n+1}} \ \ \mathrm{and} \ \ \| (\tilde{v}w u-u_n)^* \xi_j\| <\frac{1}{2^{n+1}}, \ \ 1\le j\le n+1.
\end{align}

Define $\theta_{n+1}:=\Ad(\tilde{v}w)\circ \theta$ and $u_{n+1}:= \tilde{v}w u$.
By (\ref{4.2.5}), condition (5)$_{n+1}$ is trivial.
Since $\tilde{v}w\in A$ and 
\[
\| \tilde{v}w-u_{n+1}\|=\|\tilde{v}w-\tilde{v}w u\|=\|I-u\|\le 75\gamma,
\]
condition (9)$_{n+1}$ holds.

By Lemma \ref{4.1} (iv),
$\theta_{n+1}=\Ad(\tilde{v}w)\circ\theta\approx_{Y_{n+1},\delta_{n+1}}\Ad(\tilde{v}w u)=\Ad (u_{n+1})$.
Thus, condition (4)$_{n+1}$ is satisfied.

Fix $x\in X_n$.
Let $y\in Y_{n+1}$ satisfy $\| x- y\|\le 2^{-(n+1)}/3$.
Then, by (4)$_{n+1}$,
\begin{align*}
&\| \theta_{n+1}(x)- \Ad (u_{n+1}) (x) \| \\
&\le \| \theta_{n+1}(x) -\theta_{n+1}(y)\|+\|\theta_{n+1}(y)-\Ad (u_{n+1}) (y)\| 
 +\| \Ad (u_{n+1}) (y)- \Ad (u_{n+1}) (x) \| \\
&\le \frac{1}{3\cdot 2^{n+1}}+\delta_{n+1}+\frac{1}{3\cdot 2^{n+1}} < \frac{1}{2^{n+1}}.
\end{align*}
Therefore,
\begin{align*}
\theta_{n+1}\approx_{X_n,2^{-(n+1)}} \Ad (u_{n+1}).
\end{align*}
This and (\ref{4.2.4}) give $\theta_{n+1}\approx_{X_n,2^{-n}}\theta_n$.
Hence, condition (3)$_{n+1}$ holds.

For any $x\in A_1$,
\begin{align}\label{4.2.6}
\begin{split}
&\| \Ad(\tilde{v}w)(x)- \Ad (z)(x) \| \\
&\le \| \Ad(\tilde{v}w)(x)- \Ad (w) (x)\| + \| \Ad (w) (x)- \Ad (u_n) (x)\|  
 + \| \Ad (u_n) (x)-  \Ad (z) (x)\| \\
&\le 2\|\tilde{v}-I\|+ 2\| w-u_n\|+ 2\|u_n-z\| \\
&\le (1480+ 1330+ 150)\gamma=2960\gamma.
\end{split}
\end{align}
For $1\le j\le n+1$,
there exists $x_j\in X_{n+1}$ such that $\| x_j- z^* a_j z\| \le \gamma$
by the definition of $X_{n+1}$.
(\ref{4.2.6}) and Lemma \ref{4.1} (v) give
\begin{align*}
&\| \theta_{n+1}(x_j) - a_j\| \\
&\le \| \theta_{n+1}(x_j)-\Ad(\tilde{v}w)(x_j)\|+\|\Ad(\tilde{v}w)(x_j)-\Ad (z) (x_j)\| 
 +\|\Ad (z) (x_j)-a_j\| \\
&\le \| (\Ad(\tilde{v}w)\circ\theta)(x_j)-\Ad(\tilde{v}w)(x_j)\| + 2960\gamma + \| x_j - z^* b_j z\| \\
&\le \| \theta(x_j) - x_j\| + 2960\gamma+\gamma \\
&\le 3078\gamma <\frac{9}{10}.
\end{align*}
Therefore, condition (6)$_{n+1}$ is proved.

Let $\phi\colon B\to A$ be a $C$-fixed surjective $*$-isomorphism with $\phi^{-1}\approx_{Z_{n+1},365\gamma}\id_A$.
By Lemma \ref{4.1} (vi),
there exists a unitary $\tilde{w}\in C'\cap A$ such that 
\begin{align}\label{4.2.7}
\Ad (\tilde{w}) \circ \phi \approx_{Y_{n+1},\delta_{n+1}/2} \theta
\end{align}
and $\| \tilde{w}- u \| \le 665\gamma$.
For any $y\in Y_{n+1}$, by (\ref{4.2.7}),
\begin{align*}
\| (\Ad(\tilde{v}w \tilde{w})\circ \phi)(y) - \theta_{n+1}(y) \| 
=\| (\Ad (\tilde{w}) \circ \phi)(y) - \theta(y) \| \le \frac{\delta_{n+1}}{2}.
\end{align*}
Furthermore, we have
\begin{align*}
\| \tilde{v}w\tilde{w}-u_{n+1}\| =\| \tilde{v}w\tilde{w} - \tilde{v}w u \| =\| \tilde{w}-u\| \le 665 \gamma.
\end{align*}
Thus, $\tilde{v}v\tilde{w}$ satisfies  (7)$_{n+1}$.

It remains to prove condition (8)$_{n+1}$.
Let $S\subseteq H_1$ be a finite set and $v\in C'\cap C^*(A,B)$ be a unitary with 
$\| v-u_{n+1}\|\le 740\gamma$ and $\Ad (v) \approx_{Y_{n+1},\delta_{n+1}}\theta_{n+1}$.
Then, we have
\[
\| w^*\tilde{v}^*v - u\|=\| v-\tilde{v}w u\| =\| v- u_{n+1}\|\le 740\gamma
\]
and 
$\Ad(w^*\tilde{v}^*v)\approx_{Y_{n+1},\delta_{n+1}}
\Ad(w^*\tilde{v}^*)\circ\theta_{n+1} =\theta$.
Hence,
by applying  Lemma \ref{4.1} (vii) to 
$w^*\tilde{v}^*v$ and $S':=S\cup \{ w^*\tilde{v}^*\xi : \xi\in S\}$, 
there exists a unitary $v'\in C'\cap A$ such that
$\Ad(v' w^*\tilde{v}^*v)\approx_{X_{n+1},\delta_n/6}\theta$, $\| v'-I\|\le 740\gamma$ and
\begin{align*}
\| (v' w^*\tilde{v}^*v- u) \xi'\|<\frac{1}{2^{n+2}} \ \ \mathrm{and} \ \ 
\| (v' w^*\tilde{v}^*v- u)^* \xi'\|<\frac{1}{2^{n+2}}, \ \ \xi'\in S'.
\end{align*}
For any $x\in X_n$, we have
\begin{align*}
\| \Ad(\tilde{v}w v' w^* \tilde{v}^* v)(x)- \theta_{n+1}(x) \| 
=\| \Ad( v' w^* \tilde{v}^* v)(x) - \theta(x) \|\le \frac{\delta_n}{6}<\frac{1}{2^{n+2}}.
\end{align*}
and
\begin{align*}
\| \tilde{v}w v' w^* \tilde{v}^* -I\| =\| v'-I \| \le 740\gamma.
\end{align*}
For $\xi\in S$, we have
\begin{align*}
\| ( \tilde{v}w v' w^* \tilde{v}^* v - u_{n+1})\xi \| 
= \| (v' w^* \tilde{v}^* v - u)\xi\| < \frac{1}{2^{n+2}}.
\end{align*}
and
\begin{align*}
\| (\tilde{v} w v' w^* \tilde{v}^* v - u_{n+1})^*\xi \|
=\| (v' w^*\tilde{v} v - u )^* w^* \tilde{v}^* \xi\| <\frac{1}{2^{n+2}}.
\end{align*}
Therefore, $\tilde{v}w v' w^*\tilde{v}^*$ satisfies $(8)_{n+1}$, and the lemma follows.
\end{proof}

\begin{prop}\label{4.3}
Let $C\subseteq D$ be a unital inclusion of $\mathrm{C}^*$-algebras 
acting non-degenerately on a separable Hilbert space $H$.
Let $A$ and $B$ be separable intermediate $\C^*$-subalgebras for $C\subseteq D$ with 
a conditional expectation $E\colon D\to B$.
Suppose that 
$C\subseteq A$ is crossed product-like by a discrete amenable group
and $C'\cap C^*(A,B)\subseteq \overline{C'\cap A}^{\w}$.
If $\dab<10^{-5}$,
then there exists a unitary $u\in C'\cap (A\cup B)''$ such that $u A u^* = B$.
\end{prop}

\begin{proof}
Let $\{ a_n\}_{n=1}^{\infty}$, $\{ b_n\}_{n=1}^{\infty}$ and $\{ \xi_n\}_{n=0}^{\infty}$ be dense subsets in $A_1$, $B_1$ and $H_1$, respectively.
In Lemma \ref{4.2},
we may choose $\{X_n\}_{n=0}^{\infty}$, $\{Y_n\}_{n=0}^{\infty}$, $\{Z_n\}_{n=0}^{\infty}$,
$\{\delta_n\}_{n=0}^{\infty}$, $\{u_n\}_{n=0}^{\infty}$ and
$\{\theta_n\}_{n=0}^{\infty}$ with (1)--(8).
For any $b_k$ and $\varepsilon>0$,
there is $N\in\mathbb{N}$
such that $2^{-(N-1)}<\varepsilon$ and $k< N$.
For $m \ge n \ge N$,
\[
\| \theta_m(b_k)-\theta_n(b_k)\| \le \sum_{j=n}^{m-1} \| \theta_{j+1}(b_k) - \theta_j(b_k) \| \le\sum_{j=n}^{m-1} \frac{1}{2^j} <\frac{1}{2^{N-1}} <\varepsilon.
\]
Thus, for any $b_k$, $\{\theta_n(b_k)\}_{n=0}^{\infty}$ is a Cauchy sequence. 
Since $\| \theta_n\| \le 1$, the sequence $\{\theta_n\}$ converges to a $C$-fixed $*$-homomorphism 
$\theta\colon B\to A$ in the point-norm topology.

For any $a_j$ and $n\ge j$, there is $x\in X_n$ such that $\| \theta_n(x) - a_j\| \le 9/10$.
\begin{align*}
\| a_j - \theta(x) \| 
&\le \| a_j - \theta_n(x) \| + \sum_{m=n}^{\infty} \| \theta_{m+1}(x) - \theta_m(x) \| \\
&\le \frac{9}{10}+ \sum_{m=n}^{\infty}\frac{1}{2^m} \le \frac{9}{10}+\frac{1}{2^{n-1}}.
\end{align*}
Since $n\ge j$ was arbitrary and $\{a_n\}$ is dense set in $A_1$, we have $d(A,\theta(B))<1$.
Therefore, $\theta$ is surjective by Corollary \ref{surjective}. 

By Lemma \ref{4.2} (5), $\{u_n\}$ converges to a unitary $u\in C'\cap (A\cup B)''$ in the $*$-strong topology.
Moreover, by Lemma \ref{4.2} (4), we have $\theta=\Ad (u)$.
Therefore, $A=u B u^*$, since $\theta$ is surjective.
\end{proof}

Finally, we show Theorem D by using Proposition \ref{4.3} and Corollary \ref{2.4}.

\begin{thm}\label{main}
Let $C\subseteq D$ be a unital inclusion of $\mathrm{C}^*$-algebras 
acting on a separable Hilbert space $H$.
Let $A$ and $B$ be separable intermediate $\C^*$-subalgebras for $C\subseteq D$ with 
a conditional expectation $E\colon D\to B$.
Suppose that $C\subseteq A$ is crossed product-like by a discrete amenable group
and $C'\cap A$ is weakly dense in $C'\cap \overline{A}^{\w}$.
If $\dab<10^{-7}$,
then there exists a unitary $u\in C'\cap (A\cup B)''$ such that $u A u^* = B$.
\end{thm}

\begin{proof}
Let $\dab<\gamma<10^{-7}$.
By Corollary \ref{2.4},
there exists a unitary $u_0\in (C^{**})'\cap W^*(A^{**}, B^{**})$ such that
$u_0 A^{**} u_0^* = B^{**}$ and $\| u_0-I\| \le 19\gamma$.

Let $e_D$ be the support projection of $D$ and define $K:=\mathrm{ran}(e_D) \subseteq H$.
Now restrict $A,B,C$ and $D$ to $K$.
By the universal property, there exists a unique normal representation $\pi\colon D^{**}\to \mathbb{B}(K)$ such that 
$\pi|_D=\id_D$ and $\pi(D^{**})=D''$.

Define $\tilde{A}:= \pi(u_0) A \pi(u_0^*)\subseteq \mathbb{B}(K)$,
then 
$d(\tilde{A},B)\le 2\|u_0-I\| + \dab<39\gamma<10^{-5}$.
Since $\tilde{A}'' = \pi(u_0) \pi(A^{**}) \pi(u_0^*)=\pi(B^{**}) = B''$
and $C'\cap A$ is weakly dense in $C'\cap \overline{A}^{\w}$,
\[
C'\cap C^*(\tilde{A},B)\subseteq C'\cap \tilde{A}'' = \pi(u_0) ( C'\cap A'')\pi(u_0)^* = \pi(u_0) (\overline{C'\cap A}^{\w}) \pi(u_0)^*
=\overline{C'\cap \tilde{A}}^{\w}.
\]
Therefore,
there exists a unitary $u_1\in C'\cap B''\subseteq \mathbb{B}(K)$ such that $u_1 \tilde{A} u_1^* = B$ by Proposition \ref{4.3}.
Hence, the unitary $u$ is given by 
\[
u=u_1\pi(u_0)+(I_K-e_D) \in C'\cap (A\cup B)''\subseteq \mathbb{B}(H),
\]
so that $u A u^*=B$.
\end{proof}

\begin{exam}\upshape
Let $C=C(\mathbb{T})$ and $A=C(\mathbb{T})\rtimes \mathbb{Z}$ act on 
$H=\mathcal{L}^2(\mathbb{T})\otimes \ell^2(\mathbb{Z})$.
Then we have $C'\cap A=C$ and $C'\cap \overline{A}^{\w}=\mathcal{L}^{\infty}(\mathbb{T})$,
that is, $C'\cap A$ is weakly dense in $C'\cap \overline{A}^{\w}$.
\end{exam}

But we should be careful that $C'\cap \overline{A}^{\w}$ may not
be equal to the weak closure of $C'\cap A$ in general.

\begin{exam}\upshape
Let $\alpha$ be a free action of a group $G$ on a simple C$^*$-algebra $C$
and $A=C\rtimes_{\alpha}G$ act irreducibly on a Hilbert space $H$.
Then $C'\cap A=\mathbb{C}$ but $C'\cap \overline{A}^{\w}=C'\cap \mathbb{B}(H)$.
\end{exam}

\renewcommand{\defn}{{\bf Acknowledgment.}}

\begin{defn}
The author would like to thank Professor Yasuo Watatani for his encouragement and advise.
\end{defn}

\end{document}